\documentclass[11 pt]{amsart}

\usepackage{amsfonts}
\usepackage{amsmath,amscd}
\usepackage{fullpage}
\usepackage{amssymb}
\usepackage{pb-diagram} 
\usepackage{mathrsfs}
\usepackage{bbm}
\usepackage{mathabx}

\usepackage[color]{showkeys}

\definecolor{refkey}{rgb}{1,1,1}
\definecolor{labelkey}{rgb}{1,1,1}
\definecolor{cite}{rgb}{0.9451,0.2706,0.4941}
\definecolor{ruri}{rgb}{0.0078,0.4022,0.8010}

\usepackage[%
bookmarks=true,bookmarksnumbered=true,%
colorlinks=true,linkcolor=ruri,citecolor=red%
 ]{hyperref}

\makeindex 

\def\F{{\rm \mathbb{F}}}
\def\FF{{\rm \mathcal{F}}}
\def\NN{{\rm \mathbb{N}}}
\def\Z{{\rm \mathbb{Z}}}
\def\N{{\rm \textbf{N}}}
\def\Q{{\rm \mathbb{Q}}}
\def\C{{\rm \mathbb{C}}}
\def\R{{\rm \mathbb{R}}}
\def\T{{\rm \mathbb{T}}}
\def\V{{\rm \mathbb{V}}}

\def\A{{\rm \mathbb{A}}}
\def\X{{\rm \mathcal{X}}}
\def\AA{{\rm \mathcal{A}}}
\def\BB{{\rm \mathcal{B}}}
\def\Aa{{\rm \mathscr{A}}}
\def\W{{\rm \mathcal{W}}}
\def\D{{\rm \mathcal{D}}}
\def\L{{\rm \mathcal{L}}}
\def\B{{\rm \mathscr{B}}}
\def\E{{\rm \mathcal{E}}}
\def\p{{\rm \mathfrak{p}}}
\def\m{{\rm \mathfrak{m}}}

\def\a{{\rm \mathfrak{a}}}
\def\n{{\rm \mathfrak{n}}}
\def\f{{\rm \mathfrak{f}}}
\def\d{{\rm \mathfrak{d}}}
\def\s{{\rm \sigma}}
\def\triv{{\rm \mathbbm{1}}}

\def\b{{\rm \mathfrak{b}}}
\def\c{{\rm \mathfrak{c}}}

\def\O{{\rm \mathcal{O}}}
\def\CC{{\rm \mathcal{C}}}

\def\into{{\rm \hookrightarrow}}

\def\Aut{{\rm Aut}}
\def\Sym{{\rm Sym}}
\def\End{{\rm End}}
\def\ur{{\rm ur}}
\def\res{{\rm res}}
\def\rig{{\rm rig}}
\def\Tr{{\rm Tr}}
\def\et{{\rm et}}
\def\ab{{\rm ab}}
\def\cyc{{\rm cyc}}
\def\cor{{\rm cor}}
\def\ord{{\rm ord}}
\def\cris{{\rm cris}}
\def\tors{{\rm tors}}
\def\CH{{\rm CH}}
\def\Ext{{\rm Ext}}
\def\Nm{{\rm Nm}}
\def\lcm{{\rm lcm}}

\def\conv{{\rm conv}}
\def\rel{{\rm \hspace{1mm} rel\hspace{1mm}}}\def\e{{\rm \epsilon}}
\def\l{{\rm \lambda}}
\def\GL{{\rm GL}}
\def\Gal{{\rm Gal}}
\def\Frob{{\rm Frob}}

\def\CH{{\rm CH}}
\def\Corr{{\rm Corr}}
\def\SL{{\rm SL}}
\def\Pic{{\rm Pic}}
\def\Hom{{\rm Hom}}

\def\Spec{{\rm Spec \hspace{1mm}}}

\def\id{{\rm id}}

\def\mod{{\rm mod \hspace{1 mm}}}

\numberwithin{equation}{section}

\newtheorem{theorem}{Theorem}
\newtheorem{lemma}[theorem]{Lemma}

\newtheorem{corollary}[theorem]{Corollary}
\newtheorem{proposition}[theorem]{Proposition}
\newenvironment{definition}[1][Definition]{\begin{trivlist}
\item[\hskip \labelsep {\bfseries #1}]}{\end{trivlist}}

\newenvironment{remark}[1][Remark]{\begin{trivlist}
\item[\hskip \labelsep {\bfseries #1}]}{\end{trivlist}}

\begin{document}

\title{$p$-adic heights of generalized Heegner cycles}

\date{\today}

\author{Ariel Shnidman}
\address{Department of Mathematics, Boston College, Chestnut Hill, MA 02467-3806}
\email{shnidman@bc.edu}

\begin{abstract}
We relate the $p$-adic heights of generalized Heegner cycles to the derivative of a $p$-adic $L$-function attached to a pair $(f, \chi)$, where $f$ is an ordinary weight $2r$ newform and $\chi$ is an unramified imaginary quadratic Hecke character of infinity type $(\ell,0)$, with $0 < \ell < 2r$.  This generalizes the $p$-adic Gross-Zagier formula in the case $\ell = 0$ due to Perrin-Riou (in weight two) and Nekov\'a\v{r} (in higher weight).
\end{abstract}

\maketitle
\tableofcontents
\section{Introduction}

Let  $p$ be an odd prime, $N \geq 3$ an integer prime to $p$, and $f = \sum a_n q^n$ a newform of weight $2r> 2$ on $X_0(N)$ with $a_1 = 1$.  Fix embeddings $\bar \Q \to \C$ and $\bar \Q \to \bar\Q_p$ once and for all, and suppose that $f$ is ordinary at $p$, i.e. the coefficient $a_p\in \bar \Q_p$ is a $p$-adic unit.  Building on work of Perrin-Riou \cite{PR1}, Nekov\'a\v{r} \cite{Nek} proved a $p$-adic analogue of the Gross-Zagier formula \cite{GZ} for $f$ along with any character $\CC: \Gal(H/K) \to \bar \Q^\times$.  Here, $K$ is an imaginary quadratic field of odd discriminant $D$ such that all primes dividing $pN$ split in $K$, and $H$ is the Hilbert class field of $K$.  

Nekov\'a\v{r}'s formula relates the $p$-adic height of a Heegner cycle to the derivative of a $p$-adic $L$-function attached to the pair $(f, \CC)$.  Together with the Euler system constructed in \cite{NekEuler}, the formula implies a weak form of Perrin-Riou's conjecture \cite[Conj.\ 2.7]{colmez}, a $p$-adic analogue of the Bloch-Kato conjecture for the motive $f \otimes K$ \cite[Theorem B]{Nek}.  The connection between special values of $L$-functions and algebraic cycles is part of a very general (conjectural) framework articulated in the works of  Beilinson, Bloch, Kato, Perrin-Riou, and others.  Despite the fact that these conjectures can be formulated for arbitrary motives, they have been verified only in very special cases.   

The goal of this paper is to extend the ideas and computations in \cite{Nek} to a larger class of motives.  Specifically, we will consider motives of the form $f \otimes\,  \Theta_\chi$, where 
\[\chi:\A_K^\times/K^\times \to \C^\times\]
is an unramified Hecke character of infinity type $(\ell,0)$, with $0 < \ell = 2k < 2r$, and 
\[\Theta_\chi = \sum_{\a \subset \O_K} \chi(\a)q^{\N\a}\]
is the associated theta series.   The conditions on $\ell$ guarantee that the Hecke character $\chi_0 := \chi^{-1}\N^{r+k}$ of infinity type $(r+k, r-k)$ is critical in the sense of \cite[\S4]{BDP1}. Note that $L(f,\chi_0^{-1}, 0) = L(f, \chi, r+k)$ is the central value of the Rankin-Selberg $L$-function attached to $f \otimes \Theta_\chi$.  If we take $\ell = 0$, then $\chi$ comes from a character of $\Gal(H/K)$, so we are in the situation of \cite{Nek}.  Our main result (Theorem \ref{main}) extends Nekov\'a\v{r}'s formula to the case $\ell >  0$ by relating $p$-adic heights of \textit{generalized} Heegner cycles to the derivative of a $p$-adic $L$-function attached to the pair $(f,\chi)$.  We now describe both the algebraic cycles and the $p$-adic $L$-function needed to state the formula.  

\subsection{Generalized Heegner cycles}
Let $Y(N)/\Q$ be the modular curve parametrizing elliptic curves with full level $N$ structure, and let $\E \to Y(N)$ be the universal elliptic curve with level $N$ structure.  Denote by $W = W_{2r-2}$, the canonical non-singular compactification of the $(2r-2)$-fold fiber product of $\E$ with itself over $Y(N)$ \cite{Sch}.  Finally, let $A/H$ be an elliptic curve with complex multiplication by $\O_K$ and good reduction at primes above $p$.   We assume further that $A$ is isogenous (over $H$) to each of its $\Gal(H/K)$-conjugates $A^\sigma$ and that $A^\tau \cong A$, where $\tau$ is complex conjugation.  Such an $A$ exists since $K$ has odd discriminant \cite[\S 11]{Gr}.  Set $X = W_H \times_H A^\ell$, where $W_H$ is the base change to $H$.  $X$ is fibered over the compactified modular curve $X(N)_H$, the typical geometric fiber being of the form $E^{2r-2} \times A^\ell$, for some elliptic curve $E$.  

The $(2r + 2k - 1)$-dimensional variety $X$ contains a rich supply of \textit{generalized} Heegner cycles supported in the fibers of $X$ above Heegner points on $X_0(N)$ (we view $X$ as fibered over $X_0(N)$ via $X(N) \to X_0(N)$).  These cycles were first introduced by Bertolini, Darmon, and Prasanna in \cite{BDP1}.  In Section \ref{cycles}, we define certain cycles $\e_B\epsilon Y$ and $\e_B\bar\e Y$ in $\CH^{r+k}(X)_K$ which sit in the fiber above a Heegner point on $X_0(N)(H)$, and which are variants of the generalized Heegner cycles which appear in \cite{BDP3}.  Here, $\CH^{r+k}(X)_K$ is the group of codimension $r+k$ cycles on $X$ with coefficients in $K$ modulo rational equivalence.  In fact, for each ideal $\a$ of $K$, we define cycles $\e_B\e Y^\a$ and $\e_B\bar\e Y^\a$ in $\CH^{r+k}(X)_K$, each one sitting in the fiber above a Heegner point.  These cycles are replacements for the notion of $\Gal(H/K)$-conjugates of $\e_B\e Y$ and $\e_B\bar\e Y$.  The latter do not exist as cycles on $X$, as $X$ is not (generally) defined over $K$.  In particular, we have $\e_B\e Y^{\O_K} =\e_B \e Y$.  

The cycles $\e_B\e Y^\a$ and $\e_B\bar\e Y^\a$ are homologically trivial on $X$ (Corollary \ref{homtriv}), so they lie in the domain of the $p$-adic Abel-Jacobi map
$$\Phi: \CH^{r+k}(X)_{0,K} \to H^1(H,V),$$
where $V$ is the $\Gal(\bar H/H)$-representation  $H_\et^{2r+2k-1}(\bar X, \Q_p)(r+k)$.  We will focus on a particular 4-dimensional $p$-adic representation $V_{f, A, \ell}$, which admits a map 
$$H_\et^{2r+2k-1}(\bar X, \Q_p)(r+k) \to V_{f,A,\ell}.$$
$V_{f,A,\ell}$ is a $\Q_p(f)$-vector space, where $\Q_p(f)$ is the field obtained by adjoining the coefficents of $f$.  As a Galois representation, $V_{f,A,\ell}$ is ordinary (Theorem \ref{ord}) and is closely related to the $p$-adic realization of the motive $f \otimes \Theta_\chi$ (see Section \ref{cycles}).  After projecting, one obtains a map
$$\Phi_f: \CH^{r+k}(X)_{0,K}  \to H^1(H,V_{f,A,\ell}),$$ which we again call the Abel-Jacobi map.  For any ideal $\a$ of $K$, define $z_f^\a = \Phi_f(\e_B\e Y^\a)$ and $\bar z_f^\a = \Phi_f(\e_B\bar \e Y^\a)$.

One knows that the image of $\Phi_f$ lies in the Bloch-Kato subgroup 
\[H^1_f(H,V_{f,A,\ell}) \subset H^1(H,V_{f,A,\ell})\] 
(Theorem \ref{AJ}).  If we fix a continuous homomorphism $\ell_K: \A^\times_K/K^\times \to \Q_p$, then \cite{Nekhts} provides a symmetric $\Q_p(f)$-linear height pairing
$$\langle \, \,, \, \rangle_{\ell_K} : H^1_f(H, V_{f,A,\ell}) \times H^1_f(H,V_{f,A,\ell}) \to \Q_p(f).$$
We can extend this height pairing $\bar \Q_p$-linearly to $H^1_f(H, V_{f,A,\ell}) \otimes \bar \Q_p$.  The cohomology classes $\chi(\a)^{-1}z_f^\a$ and $\bar\chi(\a)^{-1}\bar z_f^\a$ depend only on the class $\AA$ of $\a$ in the class group $\Pic(\O_K)$, of size $h = h_K$.  We denote the former by $z_{f,\chi}^\AA$ and the latter by $z_{f,\bar\chi}^\AA$.  Finally, set
\begin{equation*}z_{f,\chi} = \frac{1}{h}\sum_{\AA \in \Pic(\O_K)} z_{f,\chi}^\AA \hspace{5mm} \mbox{and} \hspace{5mm} z_{f,\bar \chi} = \frac{1}{h}\sum_{\AA \in \Pic(\O_K)} z_{f,\bar\chi}^\AA,\end{equation*}
both being elements of $H^1_f(H,V_{f,A,\ell}) \otimes \bar \Q_p.$  Our main theorem relates $\langle z_{f,\chi}, z_{f, \bar \chi}\rangle_{\ell_K}$ to the derivative of a $p$-adic $L$-function which we now describe.  

\subsection{The $p$-adic $L$-function}
Recall, if $f =\sum a_n q^n \in M_j(\Gamma_0(M), \psi)$ and $g = \sum b_n q^n \in M_{j'}(\Gamma_0(M), \xi)$, then the Rankin-Selberg convolution is  
$$L(f,g,s) = L_M(2s+2 - j - j', \psi\xi)\sum_{n\geq 1} a_nb_nn^{-s},$$ 
where $$L_M(s, \psi\xi) = \prod_{p \not \divides M} \left(1 - (\psi\xi)(p)p^{-s}\right)^{-1}.$$
Let $K_\infty / K$ be the $\Z_p^2$-extension of $K$ and let $K_p$ be the maximal abelian extension of $K$ unramified away from $p$.  In Section 2, we define a $p$-adic $L$-function $L_p(f \otimes \chi)(\lambda)$, which is a $\bar \Q_p$-valued function of continuous characters $\lambda: \Gal(K_\infty/K) \to 1 + p\Z_p$.  The function $L_p(f \otimes \chi)$ is the restriction of an analytic function on $\Hom(\Gal(K_p/K), \C_p^\times)$, which is characterized by the following interpolation property: if $\W : \Gal(K_p/K) \to \C_p^\times$ is a finite order character of conductor $\f$, with $\N\f = p^\beta$, then 

$$L_p(f \otimes\, \chi)(\W) = C_{f,k}\W(N)\overline{\chi\W}(\D)\tau(\chi\W)V_p(f,\chi,\W)L(f,\Theta_{\overline{\chi \W}}, r + k)$$
with $$C_{f,k} = \frac{2(r-k-1)!(r+k-1)!}{(4\pi)^{2r}\alpha_p(f)^\beta\langle f, f\rangle_N},$$
and where $\alpha_p(f)$ is the unit root of $x^2 - a_p(f) x + p^{2r-1}$, $\langle f, f\rangle_N$ is the Petersson inner product, $\D = \left(\sqrt D\right)$ is the different of $K$, $\Theta_{\overline{\chi\W}}$ is the theta series $$\Theta_{\overline{\chi\W}}\, = \sum_{(\a,\f) = 1}\overline{\chi\W}(\a)q^{\N\a},$$ $\tau(\chi\W)$ is the root number for $L(\Theta_{\chi\W}, s)$, and  
$$V_p(f,\chi,\W) = \prod_{\p | p} \left(1 - \frac{(\bar\chi\bar\W)(\p)}{\alpha_p(f)}\N(\p)^{r-k-1}\right)
\left(1 - \frac{(\chi\W)(\p)}{\alpha_p(f)}\N(\p)^{r-k-1}\right).$$

Recall we have fixed a continuous homomorphism $\ell_K: \A^\times_K/K^\times \to \Q_p$.  Thinking of $\ell_K$ as a map $\Gal(K_\infty/K) \to \Q_p$, we may write $\ell_K = p^{-n}\log_p \circ \lambda$, for some continuous $\lambda: \Gal(K_\infty/K) \to 1+p\Z_p$ and some integer $n$.  The derivative of $L_p$ at the trivial character in the direction of $\ell_K$ is by definition
$$L_p'(f \otimes \chi, \ell_K,\mathbbm{1})  = p^{-n}\frac{d}{ds}L_p(f \otimes \chi)(\lambda^s)\bigg|_{s=0}.$$ 
With these definitions, we can finally state our main result.  
\begin{theorem}\label{main}
If $\chi$ is an unramified Hecke character of $K$ of infinity type $(\ell, 0)$ with $0 < \ell = 2k < 2r$, then
$$L_p'(f \otimes \chi,\ell_K, \mathbbm{1}) =(-1)^k  \prod_{\p | p}\left(1 - \frac{\chi(\p)p^{r-k-1}}{\alpha_p(f)} \right)^2\frac{h\left\langle z_{f,\chi}, z_{f,\bar\chi}\right\rangle_{\ell_K}}{u^2\left(4|D|\right)^{r-k-1}},$$
where $h = h_K$ is the class number and $u = \frac{1}{2}\O_K^\times$.  
\end{theorem}

\begin{remark}\label{realj}
Our assumption that $A^\tau \cong A$ implies that the lattice corresponding to $A$ is 2-torsion in the class group. This is convenient for proving the vanishing of the $p$-adic height in the anti-cyclotomic direction, but not strictly necessary.  One should be able to prove the theorem without this assumption by making use of the functoriality of the height pairing to relate heights on $X$ to heights on $X^\tau$, but we omit the details.            
\end{remark}

\begin{remark}
When $\ell = 0$ the cycles and the $p$-adic $L$-function simplify to those constructed in \cite{Nek}, and the main theorem becomes Nekov\'a\v{r}'s formula, at least up to a somewhat controversial sign.  It appears that a sign was forgotten in \cite[II.6.2.3]{Nek}, causing the discrepancy with our formula and with Perrin-Riou's as well.  Perrin-Riou's formula \cite{PR1} covers the case $\ell = 0$ and $r = 1$. 
\end{remark}


\begin{remark}
We have assumed $N \geq 3$ for the sake of exposition.  For $N < 3$, the proof should be modified to account for the lack of a fine moduli space and extra automorphisms in the local intersection theory.  These details are spelled out in \cite{Nek} and pose no new problems.   
\end{remark}

\begin{remark}
There should be an archimedean analogue of Theorem \ref{main}, generalizing Zhang's formula for Heegner cycles \cite{Zhang} to the `generalized' situation.  The author plans to present such a result in the near future.      
\end{remark}

\subsection{Applications}\label{apps}
Theorem \ref{main} implies special cases of Perrin-Riou's $p$-adic Bloch-Kato conjecture.  The assumption that $A$ is isogenous to all its $\Gal(H/K)$-conjugates implies that the Hecke character 
\[\psi_H: \A^\times_H \to \C^\times,\] which is attached to $A$ by the theory of complex multiplication, factors as $\psi_H = \psi \circ \Nm_{H/K}$, where $\psi$ is a $(1,0)$-Hecke character of $K$.  Assume for simplicity that $\chi = \psi^\ell$, and set $\chi_H = \psi_H^\ell$ and $G_H := \Gal(\bar H/H)$.  Then the $G_H$-representation $V_{f,A,\ell}$ is the $p$-adic realization of a Chow motive $M(f)_H \otimes M(\chi_H)$.  Here, $M(f)$ is the motive over $\Q$ attached to $f$ by Deligne, and $M(\chi_H)$ is a motive over $H$ (with coefficients in $K$) cutting out a two dimensional piece of the middle degree cohomology of $A^\ell$.  In fact, the motive $M(\chi_H)$ descends to a motive $M(\chi)$ over $K$ with coefficients in $\Q(\chi)$ (see Remark \ref{descend}).  We write $V_{f,\chi}$ for the $p$-adic realization of $M(f)_K \otimes M(\chi)$, so that $V_{f,\chi}$ is a $G_K$-representation whose restriction to $G_H$ is isomorphic to $V_{f,A,\ell}$.  In fact, $V_{f,\chi} \cong \chi \oplus \bar\chi$, where we now think of $\chi$ as a $\Q(\chi) \otimes \Q_p$-valued character of $G_K$.  It follows that 
\[L(V_{f,\chi}, s) = L(f,\chi,s)L(f,\bar\chi,s) = L(f,\chi,s)^2.\]   

The Bloch-Kato conjecture for the motive $M(f)_K \otimes M(\chi)$ over $K$ reads 
\[ \dim H^1_f(K, V_{f,\chi}) = 2\cdot \ord_{s = r + k} L(f,\chi,s).\]
Similarly, Perrin-Riou's $p$-adic conjecture \cite[Conj.\ 2.7]{colmez} \cite[4.2.2]{PRbook} reads
\begin{equation}\label{PRconj}
 \dim H^1_f(K, V_{f,\chi}) = 2\cdot \ord_{\lambda = \triv} L(f,\chi,\ell_K, \lambda),
 \end{equation}
where $\ell_K$ is the cyclotomic logarithm and the derivatives are taken in the cyclotomic direction.  In Section \ref{proof}, we deduce the ``analytic rank 1" case of Perrin-Riou's conjecture by combining our main formula with the forthcoming results of Elias \cite{yara} on Euler systems for generalized Heegner cycles:    
\begin{theorem}\label{PRproof}
If $L'_p(f \otimes \chi, \ell_K, \mathbbm{1}) \neq 0$, then $(\ref{PRconj})$ is true, i.e.\ Perrin-Riou's $p$-adic Bloch-Kato conjecture holds for the motive $M(f)_K \otimes M(\chi)$.  
\end{theorem}

\begin{remark}
Alternatively, we can think of $z_{f,\chi}$ (resp.\ $z_{f,\bar \chi}$) as giving a class in $H^1_f(K, V_f \otimes \chi)$ (resp.\ $H^1_f(K, V_f \otimes \bar \chi)$), and note that $L(V_f \otimes \chi, s) = L(f,\chi, s) = L(V_f \otimes \bar\chi,s )$.  The Bloch-Kato conjecture for the motive $f \otimes \chi$ over $K$ then reads
\[ \dim H^1_f(K, V_f \otimes \chi) =  \ord_{s = r + k} L(f,\chi,s),\]
and similarly for $\bar\chi$ and the $p$-adic $L$-functions. 
\end{remark}


We anticipate that Theorem \ref{main} can also be used to study the variation of generalized Heegner cycles in $p$-adic families, in the spirit of \cite{francesc} and \cite{BHiwa}.  Theorem \ref{main} allows for variation in not just the weight of the modular form $f$, but in the weight of the Hecke character $\chi$ as well.

\subsection{Related work}
There has been much recent work on the connections between Heegner cycles and $p$-adic $L$-functions.  Generalized Heegner cycles were first studied in \cite{BDP1}, where their Abel-Jacobi classes were related to the special {\it value} (not the derivative) of a different Rankin-Selberg $p$-adic $L$-function.  Brooks extended these results to Shimura curves over $\Q$ \cite{hunter} and recently Liu, Zhang, and Zhang proved a general formula for arbitrary totally real fields \cite{lzz}.  In \cite{disegni}, Disegni computes $p$-adic heights of Heegner points on Shimura curves, generalizing the weight 2 formula of Perrin-Riou for modular curves.  Kobayashi \cite{kob} extended Perrin-Riou's height formula to the supersingular case.  Our work is the first (as far as we know) to study $p$-adic heights of generalized Heegner cycles.                  

\subsection{Proof outline}
The proof of Theorem \ref{main} follows \cite{Nek} and \cite{PR1} rather closely.  For this reason, we have chosen to retain much of Nekov\' a\u r's notation and not to dwell long on computations easily adapted to our situation.    

We define the $p$-adic $L$-function $L_p(f \otimes \chi, \lambda)$ in Section \ref{pL} and show that it vanishes in the anticyclotomic direction.  In Section \ref{integrate}, we integrate the $p$-adic logarithm against the $p$-adic Rankin-Selberg measure to compute what is essentially the derivative of $L_p(f\otimes \chi)$ at the trivial character in the cyclotomic direction.  In Section \ref{cycles}, we define the generalized Heegner cycles and describe Hecke operators and $p$-adic Abel-Jacobi maps attached to the variety $X$.  After proving some properties of generalized Heegner cycles, we show that the RHS of Theorem \ref{main} vanishes when $\ell_K$ is anticyclotomic.  In Section \ref{localhts} we compute the local cyclotomic heights of $z_f$ at places $v$ which are prime to $p$.  In Section \ref{ordsec}, we prove that $V_{f,A,\ell}$ is an ordinary representation.  We complete the proof of the main theorem in Section \ref{proof}, modulo the results from the final section.  

In the final section, we fix the proof in \cite[II.5]{Nek}, to complete a proof of the vanishing of the contribution coming from local heights at primes above $p$.  The key ingredient is the theory of relative Lubin-Tate groups and Theorem \ref{crysmixed}.  The latter is a result in $p$-adic Hodge theory which relies on Faltings' proof of Fontaine's $C_\cris$ conjecture.  This theorem (or rather, its proof) is quite general and should be useful for computing $p$-adic heights of algebraic cycles sitting on varieties fibered over curves.       

\subsection{Acknowledgments}

I am grateful to Kartik Prasanna for suggesting this problem and for his patience and direction.  Thanks go to Hunter Brooks for several productive conversations, and to Bhargav Bhatt, Daniel Disegni, Yara Elias, Olivier Fouquet, Adrian Iovita, Shinichi Kobayashi, Jan Nekov\'a\v{r}, and Martin Olsson for helpful correspondence.  The author was partially supported by National Science Foundation RTG grant DMS-0943832.


\section{Constructing the $p$-adic $L$-functions}\label{pL}

Recall $f \in S_{2r}(\Gamma_0(N))$ is an ordinary newform with trivial nebentypus.  As in the introduction, $\chi: \A^\times_K/K^\times \to \C^\times$ is an unramified Hecke character of infinity type $(2k,0)$ with $0 < 2k < 2r$.  For conventions regarding Hecke characters, see \cite[\S4.1]{BDP1}. All that follows will apply to $\chi$ of infinity type $(0,2k)$ with suitable modifications.  In this section, we follow \cite{Nek} and define a $p$-adic $L$-function attached to the pair $(f, \chi)$ which interpolates special values of certain Rankin-Selberg convolutions.  

\subsection{$p$-adic measures}

We use the notation of \cite{Nek} unless stated otherwise.  We construct the $p$-adic $L$-function only in the setting needed for Theorem \ref{main}; in the notation of \cite{Nek}, this means that $\Omega = 1, N_1 = N_2 = c_1 = c_2 = c = 1, N_3 = N'_3 = N, \Delta = \Delta_1 = \Delta_2 = |D|, \Delta_3 = 1,$ and $\gamma = \gamma_3 = 0$.  We begin by defining theta measures.

Fix an integer $m \geq 1$ and let $\O_m$ be the order of conductor $m$ in $K$.  Let $\a$ be proper $\O_m$-ideal whose class in $\Pic(\O_m)$ is denoted by $\AA$.  The quadratic form $$Q_\a(x) = \N(x)/\N(\a),$$ takes integer values on $\a$.  Define the measure $\Theta_\AA$ on $\Z^\times_p$ by 
\begin{equation}
\Theta_\AA (a (\mod p^\nu)) = \chi(\bar \a)^{-1}\sum_{\substack{x \in \a \\ Q_\a(x) \equiv a \, (\mod p^\nu) }}\\ \bar x^\ell  q^{Q_\a(x)}. \end{equation}  
To keep things from getting unwieldy we have omitted $\chi$ from the notation of the measure.  If $\phi$ is a function on $\Z/p^\nu\Z$ with values in a $p$-adic ring $A$, then  
\begin{equation}
\Theta_\AA(\phi) = \chi(\bar \a)^{-1}\sum_{x \in \a} \phi(Q_\a(x)) \bar x^\ell q^{Q_\a(x)} = \chi(\bar \a)^{-1} \sum_{n \geq 1} \phi(n) \rho_\a(n, \ell) q^n,
\end{equation} 
where $\rho_\a(n, \ell)$ is the sum $\sum \bar x^\ell$ over all $x \in \a$ with $Q_\a(x) = n.$  We have 
$$\rho_{\a(\gamma)}(n,\ell) = \bar \gamma^\ell \rho_\a(n,\ell),$$
for all $\gamma\in K^\times$, so that $\Theta_\AA$ is independent of the choice of representative $\a$ for the class $\AA$.  
For $\a \in \AA$, 
\begin{equation}
\chi(\bar \a)^{-1}\sum_{x \in \a} \bar x^\ell q^{Q_\a(x)} = w_m\sum_{\substack{\a' \in \AA \\ \a'\subset \O_m}}\\ \chi(\a')q^{\N(\a')} = w_m \sum_{n \geq 1} r_{\AA,\chi}(n) q^n,
\end{equation} 
since $\ell$ is a multiple of $w_m$.  The coefficients $r_{\AA,\chi}(n)$ play the role of (and generalize) the numbers $r_\AA(m)$ that appear in \cite{GZ} and \cite{Nek}.

\begin{proposition}
$\Theta_\AA(\phi)$ is a cusp form in $M_{\ell+1}(\Gamma_1(M), A)$, with $M = \lcm(|D|m^2,p^{2\nu})$.  
\end{proposition}

\begin{proof}

It is classical \cite{Ogg} that $\sum_{x \in \a} \bar x^\ell q^{Q_\a(x)}$ is a cusp form in $M_{\ell+1}(\Gamma_1(|D|m^2))$.  It follows from \cite[Proposition 1.1]{Hida1} that weighting this form by $\phi$ gives a modular form of the desired level.  
\end{proof}

For a fixed integer $C$, define the Eisenstein measures 
$$E_1(\alpha (\mod p^\nu))(z) = E_1(z,\phi_{\alpha,p^\nu}) $$ and 
$$E_1^C(\alpha(\mod p^\nu))(z) = E_1(\alpha(\mod p^\nu))(z) - CE_1(C^{-1}\alpha(\mod p^\nu))(z),$$ as in \cite[I.3.6]{Nek}.  Similarly, we define the following convolution measure on $\Z_p^\times$  
\begin{align*}
&\Phi^C_\AA ( a (\mod p^\nu)) =\\
 &H\left[ \sum _{\alpha \in (\Z/|D|p^\nu\Z)^\times} \xi(\alpha)\Theta_\AA(\alpha^2a (\mod p^\nu))(z)\delta^{r-1 - k}_1(E_1^C(\alpha (\mod |D|p^\nu))(Nz))\right],
 \end{align*}
which takes values in $\overline M_{2r}(\Gamma_0(N|D| p^\infty); \chi(\bar \a)^{-1}p^{-\delta}\Z_p)$, for some $\delta$ depending only on $r$ and $k$ \cite[Lem.\ 5.1]{Hida1}.  Here, $H$ is holomorphic projection, $\delta_1^{r-1-k}$ is Shimura's differential operator, and $\xi = \left(\frac{D}{\cdot}\right)$.  We are implicitly identifying $\Z_p$ with the ring of integers of $K_\p$ for a prime $\p$ above $p$ (which is split in $K$), so that $x^\ell \in \Z_p$ for all $x \in \a$. 
The measure $\Psi_\AA^C$ is defined by
$$\Psi_\AA^C = \frac{1}{2w_m}\Phi^C_\AA\bigg|_{2r} \mathscr{T}(|D|)_{N|D|p^\infty/Np^\infty},$$
where $$\mathscr{T}: M_{2r}\left(\Gamma_0\left(N|D|p^\infty\right), \cdot\right) \to M_{2r}\left(\Gamma_0\left(Np^\infty\right),\cdot\right)$$ is the trace map, i.e.\ the adjoint to the operator $g \mapsto |D|^{r-1} g\bigg|_{2r} \left( \begin{array}{cc}
|D| & 0 \\
0 & 1 \end{array} \right).$   

For ring class field characters $\rho: G(H_m/K) \to \overline \Q^\times$, define
$$\Phi_\rho^C = \sum_{[\AA] \in \Pic(\O_m)} \rho([\AA])^{-1}\Phi_\AA^C,$$ and similarly for $\Psi^C_\rho$.  We define $\Psi^C_{f,\rho} = L_{f_0}(\Psi_\rho^C),$ where $L_{f_0}$ is the Hida projector attached to the $p$-stabilization $$f_0 = f(z) - \frac{p^{2r-1}}{\alpha_p(f)}f(pz)$$ of $f$ (see \cite[I.2]{Nek} for its definition and properties).  Explicitly, if $g \in M_j(\Gamma_0(Np^\mu); \bar \Q)$ with $\mu \geq 1$, then   
\begin{equation}
L_{f_0(g)} = \left(\frac{p^{j/2 -1}}{\alpha_p(f)}\right)^{\mu-1} \frac{\left\langle f_0^\tau \bigg |_j \left( \begin{array}{cc}
0 & -1 \\
N p^\mu & 0 \end{array} \right), g \right\rangle_{Np^\mu}}{\left\langle f_0^\tau \bigg |_j \left( \begin{array}{cc}
0 & -1 \\
N p & 0 \end{array} \right), f_0 \right\rangle_{Np}}.
\end{equation}
We also define a measure $\Psi_f^C$ on $\Gal(H_{p^\infty}/K) \times \Gal(K(\mu_{p^\infty})/K)$ by
$$\Psi^C_f(\sigma \,(\mod p^n), \tau \, (\mod p^m)) = L_{f_0}(\Psi^C_\AA (a\, (\mod p^m))),$$
where $\sigma$ corresponds to $\AA$ and $\tau$ corresponds to $a \in (\Z/p^m\Z)^*$ under the Artin map.  Finally, as in \cite{Nek}, we define modified measures $\tilde \Psi_\AA^C, \tilde \Psi^C_\rho,$ etc., by replacing $\mathscr{T}(|D|)$ with $\mathscr{T}(1)$ in the definition of $\Psi_\AA^C$.     

\subsection{Integrating characters against the Rankin-Selberg measure}

In this subsection, we integrate finite order characters of the $\Z_p^2$-extension of $K$ against the measures constructed in the previous section and show that they recover special values of Rankin-Selberg $L$-functions.  This allows us to prove a functional equation for the (soon to be defined) $p$-adic $L$-function.   We follow the computations in \cite[I.5]{Nek} and \cite[\S4]{PR2}. Let $\eta$ denote a character $(\Z/p^\nu\Z)^\times \to \bar \Q^\times$.  
Exactly as in \cite[Lemma 7]{PR2}, we compute:
\begin{equation}\int_{\Z^\times_p} \eta \,d\Phi_\AA^C = \left(1 - C\xi(C)\bar \eta^2(C)\right)H[\Theta_\AA(\eta)(z)\delta_1^{r-k-1}\left(E_1(Nz,\phi)\right)].
\end{equation}
Similarly, if $\rho$ is a ring class character with conductor a power of $p$,
\begin{equation}
\int_{\Z^\times_p} \eta \,d\Phi_\rho^C = w_m\left(1 - C\xi(C)\bar \eta^2(C)\right)H[\Theta_\chi(\W '')(z)\delta_1^{r-k-1}\left(E_1(Nz,\phi)\right)],
\end{equation}
where $\W'' = \rho \cdot (\eta \circ \N)$, the latter being thought of as a character modulo the ideal $\f =  \lcm($cond $\rho$, cond $\eta, p)$.  We denote by $\W$ the primitive character associated to $\W''$.  By definition,  
$$\Theta_\chi(\W'')(z) = \sum_{\substack{\a \subset \O \\ (\a, \f) = 1}}\W''(\a)\chi(\a)q^{\N(\a)}.$$  This is a cusp form in $S_{\ell+1}\left(|D|\N^K_\Q\left(\f_{\W''}\right), \left(\frac{D}{\cdot}\right)\eta^2\right)$, since $\chi$ is unramified (see \cite{Ogg} for a more general result).   
The computations of \cite[I.5.3-4]{Nek} carry over to our situation, except the theta series transformation law now reads  
\begin{equation}
\Theta_\chi(\W'')(z)\bigg|_{\ell+1} \mathscr{F} = \left(\frac{D}{w}\right)\bar \eta ^2(w)\Theta_\chi(W'')\bigg|_{\ell + 1}\left( \begin{array}{cc}
0 & -1 \\
|D| p^\mu & 0 \end{array} \right),
\end{equation}
where $\mathscr{F}$ is the involution
$$ \left( \begin{array}{cc}
0 & -1 \\
N |D|p^\mu & 0 \end{array} \right) \left( \begin{array}{cc}
N & y \\
N |D|p^\mu t & N \end{array} \right)$$ with $Nxw - |D|p^\mu ty = 1$.  We then obtain
\begin{align*} &\int_{\Z_p^\times}\eta \, d\Psi_{f,\rho}^C = \\
 &\left(1 - C\xi(C)\bar \eta^2(C)\right)\frac{\left(\left(\frac{D}{\cdot}\right)\eta^2\right)(N)\lambda_N(f)|D|^{r - 1/2}}{(4\pi i)\alpha_p(f)^{-1}p^{r-1}} \frac{\Lambda_\mu(\W'')}{\left<  f_0^\tau \bigg|_{2r} \left( \begin{array}{cc}
0 & -1 \\
N p & 0 \end{array} \right), f_0\right >_{Np}},
\end{align*}
where 
\[\Lambda_\mu(\W'') = \frac{p^{\mu(r-1/2)}}{\alpha_p(f)^\mu}\left< f_0^\tau, \Theta_\chi\left(\W''\right)\bigg|_{\ell + 1}\left( \begin{array}{cc}
0 & -1 \\
|D| p^\mu & 0 \end{array} \right)\delta^{r-k-1}\left(E_1\left(z,\xi \bar \eta^2 \right)\right)\right>_{N|D| p^\mu}.\]
 We define $\tau(\chi\W)$ by the relation
\begin{equation}\Theta_\chi(\W)\big|_{\ell + 1}\left( \begin{array}{cc}
0 & -1 \\
|D|p^\beta & 0 \end{array} \right) = (-1)^{k+1}i \tau(\chi\W)\Theta_{\bar\chi}(\bar \W),\end{equation}
with $|D|p^\beta$ being the level $\Delta(\W)$ of $\Theta_\chi(\W)$.  One knows (\cite[Thm.\  4.3.12]{Miy}) that $\tau(\chi\W) \in \bar \Q^\times$, $|\tau(\chi\W)| = 1$, and $$\Lambda(\chi\W,s) = \tau(\chi\W)\Lambda(\bar\chi\bar\W,\ell+1-s),$$
where $$\Lambda(\chi\W,s) = \left(|D|p^\beta\right)^{s/2}(2\pi)^{-s}\Gamma(s)L(\Theta_\chi(\W),s).$$
Modifying the computations in \cite[\S4]{PR2}, we find that 
\begin{equation}
\Lambda_\mu(\W'') = (-1)^{k+1}i\tau(\chi\W)\sum_{\substack{\a | p \\ \N(\a) = p^s}}\mu(\a)\chi(\a)\W(\a)\Lambda_{\mu,s},\end{equation}
with
\begin{equation}\Lambda_{\mu,s} = \frac{p^{\mu\left(r-\frac{1}{2}\right) -s\left(k+\frac{1}{2}\right)} }{\alpha_p(f)^\mu} \left<f_0^\tau, \Theta_{\bar \chi}(\bar \W)\bigg|_{\ell + 1}\left( \begin{array}{cc}
p^x & 0 \\
0 & 1 \end{array} \right)\delta^{r-k-1}(E_1(z,\xi\bar \eta^2))\right>_{N|D|p^\mu}\end{equation}
and $x = \mu  - \beta -s$.


Following \cite[\S4.4]{PR2}, we compute:
\begin{align*}
&\Lambda_\mu(\W'') =\\
& (-1)^ri\tau(\chi\W)V_p(f,\chi,\W)\left(\frac{p^{r-1/2}}{\alpha_p(f)}\right)^\beta \frac{2(r+k-1)!(r-k-1)!}{(4\pi)^{2r-1}} L(f,\Theta_{\bar\chi}(\bar \W),r+k),
\end{align*}
where 
$$V_p(f,\chi,\W) = \prod_{\p | p} \left(1 - \frac{(\bar\chi\bar\W)(\p)}{\alpha_{\N(\p)}(f)}\N(\p)^{r-k-1}\right)
\left(1 - \frac{(\chi\W)(\p)}{\alpha_{\N(\p)}(f)}\N(\p)^{r-k-1}\right).$$  We have used the fact that 
\begin{equation}\label{pair}
\left\langle f^\tau, g\delta^{r-k-1}_1(E_1(z,\phi))\right\rangle_M = \frac{(1-\e(-1))(r+k-1)!(r-k-1)!}{(-1)^{r-k-1}(4\pi)^{2r-1}}L(f,g,r+k)
\end{equation}
for any $g \in S_{2k+1}(M', \epsilon)$, and where $M = M'N$.  Equation (\ref{pair}) follows from the usual unfolding trick and the fact \cite[I.1.5.3]{Nek} that
$$\delta_1^{r-k-1}(E_1(z,\phi)) = \frac{(r-k-1)!}{(-4\pi)^{r-k-1}}E_{r-k}(z,\phi).$$
We have also used the following generalization of \cite[Lemma 23]{PR2}.
\begin{lemma}
If $g$ is a modular form whose $L$-function admits a Euler product expansion $\prod_p G_p(p^{-s})$, then
$$L(f_0,g,r+k) = G_p\left(p^{r-k-1}\alpha_p(f)^{-1}\right)L(f,g,r+k).$$ 
\end{lemma}


Putting these calculations together, we obtain the following interpolation result. 
\begin{theorem}\label{interp}
For finite order characters $\W = \rho\cdot(\eta \circ \N)$ as above,
$$\left(1 - C\left(\frac{D}{C}\right)\bar \W(C)\right)^{-1}\int_{\Z_p^\times}\eta \, d\Psi^C_{f,\rho} = \frac{\mathcal{L}_p(f,\chi,\W)V_p(f,\chi,\W)\Delta(\W)^{r-1/2}}{ \alpha_p(f)^\beta H_p(f)},$$
where  
$$\mathcal{L}_p(f,\chi,\W) = \left(\frac{D}{-N}\right) \W(N)\tau(\chi\W)C(r,k)\frac{L(f,\Theta_{\bar\chi}(\bar \W),r+k)}{\left<f,f\right>_N}.$$
Here, 
$$C(r,k) = \frac{2(-1)^{r-1}(r-k-1)!(r+k-1)!}{(4\pi)^{2r}}$$ and 
$$H_p(f) = \left(1 - \frac{p^{2r-2}}{\alpha_p(f)^2}\right)\left(1 - \frac{p^{2r-1}}{\alpha_p(f)^2}\right).$$
\end{theorem}
The modified measures $\tilde \Psi_{f,\rho}^C$ satisfy
$$\int_{\Z^\times_p}\eta \, d\tilde \Psi_{f,\rho}^C = |D|^{1-r}\overline{(\chi\W)}(\mathcal{D})\int_{\Z^\times_p}\eta \,d\Psi_{f,\rho}^C,$$
where $\mathcal{D} = \left(\sqrt{D}\right)$ is the different of $K$.   
 
Now to define the $p$-adic $L$-function. Recall we have fixed an integer $C$ prime to $N|D| p$.  
\begin{definition}
For any continuous character $\phi: G(H_{p^\infty}(\mu_{p^\infty})/K) \to \bar \Q_p^\times$ with conductor of $p$-power norm, 
we define 
\begin{align*}
L_p&(f \otimes \chi, \phi) =\\
& (-1)^{r-1}H_p(f)\left(\frac{D}{-N}\right)\left(1 - C\left(\frac{D}{C}\right)\phi(C)^{-1}\right)^{-1}\int_{G(H_{p^\infty}(\mu_{p^\infty})/K) } \phi \, d \tilde\Psi^C_f.
\end{align*}
\end{definition}
The $p$-adic $L$-function $L_p(f \otimes \chi)(\lambda) := L_p(f \otimes \chi, \lambda)$ is a function
of characters $$\lambda: G(H_{p^\infty}(\mu_{p^\infty})/K)  \to (1+p\Z_p).$$  $L_p(f \otimes \chi)$ is an Iwasawa function with values in $c^{-1}\O_{\widehat{\Q(f,\chi)}}$, where $\widehat{\Q(f,\chi)}$ is the $p$-adic closure (using our fixed embedding $\bar \Q \hookrightarrow \bar \Q_p$) of the field generated by the coefficients of $f$ and the values of $\chi$, and $c \in \widehat{\Q(f,\chi)}$ is non-zero.

We can construct analogous measures and an analogous $p$-adic $L$-function for $\bar \chi$, which is a Hecke character of infinity type $(0,\ell)$.  There is a functional equation relating $L_p(f \otimes \chi)$ to $L_p(f \otimes \bar \chi),$ which we now describe.  First define 
$$\Lambda_p(f \otimes \chi)(\lambda) = \lambda(\mathcal{D}N^{-1})\lambda(N)^{1/2}L_p(f \otimes \chi)(\lambda).$$   
 
\begin{proposition}
$\Lambda_p$ satisfies the functional equation
$$\Lambda_p\left(f \otimes \chi\right)\left(\lambda\right) = \left(\frac{D}{-N}\right)\Lambda_p\left(f\otimes \bar \chi\right)\left(\lambda^{-1}\right).$$
\end{proposition} 

\begin{proof}
It suffices to prove this for all finite order characters $\W$.  For such $\W$, the functional equation for the Rankin-Selberg convolution reads
\begin{equation}\label{fxleq}L(f,\Theta_{\bar\chi}(\bar\W), r+k) = \frac{\left(\frac{D}{-N}\right)\bar \W(N)}{\tau(\chi\W)^2}L(f, \Theta_\chi(\W),r+k),\end{equation}
so 
$$\frac{\mathcal{L}_p(f,\chi,\W)}{\mathcal{L}_p(f,\bar \chi, \bar \W)} = \W(N)\left(\frac{D}{-N}\right).$$
We also have $V_p(f,\bar \chi, \bar \W) = V_p(f,\chi,\W),$ so that
$$\frac{L_p(f \otimes \chi)(\W)}{L_p(f \otimes \bar \chi)( \bar \W)} = \W(N)\left(\frac{D}{-N}\right)\bar \W(\mathcal{D})^2.$$
The proposition now follows from a simple computation.
\end{proof} 

Recall the notation $\lambda^\tau(\a) = \lambda(\a^\tau)$.  

\begin{corollary}\label{Lvanish}
Suppose $\left(\frac{D}{N}\right) = 1$ and $\lambda$ is anticyclotomic, i.e. $\lambda\lambda^\tau = 1$.  Then $L_p(f \otimes \chi)(\lambda) = 0$.  
\end{corollary}
 
\begin{proof}
From the functional equation and the fact that $$\Lambda_p(f \otimes \chi)(\lambda) = \Lambda_p(f \otimes \bar \chi)(\lambda^\tau),$$ we obtain $$\Lambda_p(f \otimes \chi)(\lambda) = -\Lambda_p(f \otimes \chi)(\lambda^{-\tau}).$$
Since $\lambda$ is anticyclotomic, this is equal to $-\Lambda_p(f \otimes \chi)(\lambda).$
\end{proof}
 
\section{Fourier expansion of the $p$-adic $L$-function}\label{integrate}
This section is devoted to computing the Fourier coefficients of $\int_{\Z_p^\times}\lambda \, d\tilde \Psi_\AA$, where $\lambda$ is a continuous function $\Z_p^\times \to \Q_p$.   These computations allow us to relate $L'_p(f \otimes \chi, \mathbbm{1})$ to heights of generalized Heegner cycles.  We follow the computations in \cite[I.6]{Nek}, however the transformation laws for theta series attached to Hecke characters complicate things a bit.  We have   
\begin{align*}
\Phi^C_{\AA} &( a (\mod p^\nu)) = \\
&H\left[ \sum _{\alpha \in (\Z/|D|p^\nu\Z)^\times} \xi(\alpha)\Theta_\AA(\alpha^2a (\mod p^\nu))(z)\delta^{r-1 - k}_1(E_1^C(\alpha (\mod |D|p^\nu))(Nz))\right],
\end{align*}    
For each factorization $D = D_1D_2$ (with the signs normalized so that $D_1$ is a discriminant), we define 
\[W_{D_1}^{(v)} = \left( \begin{array}{cc}
|D_1|a & b \\
N|D|p^\nu c & |D_1|d \end{array} \right),\] 
of determinant $|D_1|$.

\begin{lemma}\label{theta}
For $W_{D_1}^{(\nu)}$ as above and $\alpha \in \left(\Z/|D|p^\nu\Z\right)^\times$, $$\Theta_\AA\left(\alpha (\mod p^\nu)\right)(z)\bigg|_{\ell+1}W_{D_1}^{(v)} =  \frac{|D_1|^k}{\chi(\mathcal{D}_1)}
\gamma\Theta_{\AA\d_1^{-1}}\left(|D_1|a^2 \alpha (\mod p^\nu)\right)(z),$$ where $$\gamma = \left(\frac{D_1}{cp^\nu N}\right)\left(\frac{D_2}{a\N(\a)}\right)\kappa(D_1)^{-1},$$ 
 and $\mathcal{D}_1$ is the ideal of norm $|D_1|$ in $\O_K$ and $\kappa(D_1) = 1$ if $D_1 > 0$, otherwise $\kappa(D_1) = i$.
\end{lemma}
\begin{remark}
Note that the factor $\frac{|D_1|^k}{\chi(\mathcal{D}_1)}$ is equal to $\pm 1.$ 
\end{remark}
\begin{proof}
The proof proceeds as in \cite[\S3.2]{PR1}, but requires some extra Fourier analysis.  We sketch the argument for the convenience of the reader.  Fixing an ideal $\a$ in the class of $\AA$, we set $L = p^\nu \a$ and let $L^*$ be the dual lattice with the respect to the quadratic form $Q_\a$.  Denote by $S = S_\a$ the symmetric bilinear form corresponding to $Q_\a$, so $S_\a(\alpha, \beta) = \frac{1}{\N(\a)}\Tr(\alpha\bar\beta)$.  For $u \in L^*$, define
$$\Theta_{\a,\chi}(u,L) = \chi(\bar\a)^{-1}\sum_{\substack{w - u \in L\\ w \in L^*}} \bar w^\ell q^{Q_\a(w)}.$$
For any $c \in \Z$, one checks the following relations:
\begin{equation}\Theta_{\a,\chi}(u,L) = \sum_{\substack{w - u \in L\\ w \in L^*/cL}} \Theta_{\a, \chi}(w,cL),\end{equation} 
\begin{equation}\Theta_{\a,\chi}(u,cL)(c^2z) = c^{-\ell}\Theta_{\a,\chi}(cu,c^2L)(z),\end{equation}
and for all $a \in \Z$ and $w \in L^*$, 
\begin{equation}\Theta_{\a,\chi}(w,cL)\left(z + \frac{a}{c}\right) = e\left(\frac{a}{c}Q_\a(w)\right)\Theta_{\a,\chi}(w,cL).\end{equation}
We also have
\begin{equation}z^{-(\ell+1)}\Theta_{\a,\chi}(w,cL)\left(\frac{-1}{z}\right) = -ic^{-2}[L^*:L]^{-1/2} \sum_{y \in (cL)^*/cL} e\left(S_\a(w,y)\right)\Theta_{\a,\chi}(y,cL).\end{equation}
This follows from the identity
\begin{equation}z^{\ell+1}\sum_{x \in L} P(x + u)e\left(Q_\a(x+y)z\right) = i[L^*:L]^{-1/2}\sum_{y \in L^*}P(y)e\left(\frac{-Q_a(y)}{z}\right)e\left(S_\a(y,u)\right),\end{equation} valid for any rank two integral quadratic space $(L, Q_\a, S_\a)$ and any polynomial $P$ of degree $\ell$ which is spherical for $Q_\a$.  See \cite{Wall} for a proof of this version of Poisson summation.  

Now write 
$$W_{D_1}^{(\nu)} = H\left( \begin{array}{cc}
|D_1| & 0 \\
0 & 1 \end{array} \right)$$ with $H \in \SL_2(\Z)$.  Exactly as in \cite{PR1}, we use the relations above to compute
$$\Theta_{\a,\chi}(\alpha (\mod p^\nu))\bigg |_{\ell+1}H = \gamma|D_1|^{-1/2} \sum_{\substack{u \in \a/L\\ Q_\a(u) \equiv \alpha (\mod p^\nu)}}\sum_{\substack{w \in L^*/L\\ w + au \in \D_1^{-1}p^r \a}} \Theta_{\a,\chi}(w,L)$$ so that

\begin{align*}
\Theta_{\a,\chi}(\alpha (\mod p^\nu))\bigg|_{\ell+1} W_{D_1}^{(\nu)} &= \gamma|D_1|^k\chi(\bar \a)^{-1}\sum_{\substack{w \in \D_1^{-1}\a\\ Q_{\a\D_1^{-1}}(w)\equiv |D_1|a^2\alpha (\mod p^r)}}\bar w^\ell q^{Q_{\a\D_1^{-1}}(w)}\\
&= \frac{|D_1|^k}{\chi(\mathcal{D}_1)}
\gamma\Theta_{\a\D_1^{-1},\chi}\left(|D_1|a^2\alpha (\mod p^\nu)\right)(z),
\end{align*}
as desired.
\end{proof}


For any function $\lambda$ on $(Z/p^\nu\Z)^\times$, we define $h_{D_1}(\lambda)$ as in \cite[I.6.3]{Nek}, so that
$$\int_{\Z_p^\times} \lambda \, d\tilde\Psi_\AA = \frac{1}{2w} \sum_{D = D_1\cdot D_2} \sum_{j \in \Z/|D_1|\Z} h_{D_1}(\lambda)\bigg|_{2r} \left( \begin{array}{cc}
1 & j \\
0 & |D_1| \end{array} \right).$$

The Fourier coefficient computation in \cite[I.6.5]{Nek}  remains valid, except one needs to use the following proposition in place of \cite[I.1.9]{Nek}: 

\begin{proposition}\label{jacobi}
Let $f = \sum_{n \geq 1} a(n)q^n$ be a cusp form of weight $\ell + 1 = 2k + 1$, and $g = \sum_{n \geq 0}b(n)q^n$ a holomorphic modular form of weight one, both on $\Gamma_0(N)$.  Then $H(f \delta_1^{r-k-1}(g)) = \sum_{n \geq 1} c(n)q^n$ with 
\[c(n) = \frac{(-1)^{r-k-1}}{\binom{2r-2}{r-k-1}}n^{r-k-1}\sum_{i + j = n} a(i)b(j)H_{r-k-1,k}\left(\frac{i - j}{i + j}\right),\]
where 
\[H_{m,k}(t) = \frac{1}{2^m\cdot (m + 2k)!} \left(\frac{d}{dt}\right)^{m + 2k}[(t^2 - 1)^m(t - 1)^{2k}]\] 
\end{proposition}

\begin{proof}
From \cite[I.1.2.4, I.1.3.2]{Nek}, we have
\[c(n) = \frac{(r-k-1)!}{(-4\pi)^{r-k-1}}\cdot\frac{(4\pi n)^{2r-1}}{(2r-2)!} \sum_{ i + j = n}a(i)b(j)\int_0^\infty p_{r-k-1}(4\pi j y)e^{-4\pi n y} y^{r + k - 1}dy,\]
where \[p_m(x) = \sum_{a = 0}^m \binom{m}{a}\frac{(-x)^a}{a!}.\]
The integral is evaluated using the following lemma.
\begin{lemma}
Let $m, k \geq 0$.  Then 
\[\int_0^\infty p_m(4\pi j y)e^{-4\pi (i + j)y} y^{m + 2k}dy = \frac{(m + 2k)!}{(4\pi(i + j))^{m+2k+1}}H_{m,k}\left(\frac{i-j}{i+j}\right)\]
\end{lemma} 

\begin{proof}
Evaluating the elementary integrals, we find that the left hand side is equal to 
\[\frac{m!}{(4\pi (i+ j))^{m + 2k + 1}}G_{m,k}\left(\frac{j}{i + j}\right).\]
where 
\[G_{m,k}(t) = \sum_{a = 0}^m (-1)^a \frac{(m + 2k + a)!}{(a!)^2 (m - a)!} t^a.\]
It therefore suffices to prove the identity
\begin{equation}\label{combo}G_{m,k}(t) = \frac{(m+2k)!}{m!}H_{m,k}(1 - 2t).
\end{equation}
This is proved by showing that both sides satisfy the same defining recurrence relation (and base cases).  Indeed, one can check directly that for $m \geq 1$:
\begin{align}\label{recur}
(m + 1)^2(m + k)&G_{m+1,k}(t) =\\
 (2m + 2k +1)&[m^2 + m + 2km + k - (m + k)(2m+ 2k + 2)t]G_{m,k}(t)\nonumber \\ 
&- (m +k +1)(m + 2k)^2G_{m - 1,k}(t)\nonumber.
\end{align}

That the right hand side of (\ref{combo}) satisfies the same recurrence relation amounts to the well known recurrence relation for the Jacobi polynomials 
\[P_n^{(\alpha,\beta)}(t) = \frac{(-1)^n}{2^n n!}(1 - t)^{-\alpha}(1 + t)^{-\beta} \frac{d^n}{dt^n}\left[ (1 - t)^\alpha(1 + t)^\beta (1 - t^2)^n\right].\]  
Indeed, we have 
\[H_{m,k}(t) = 2^{2k}\cdot P^{(0,-2k)}_{m + 2k}(t)(1+t)^{-2k},\]
and one checks that the recurrence relation 
\begin{align*}
2(n+1)(n + \beta + 1)(2n + \beta)&P_{n+1}^{(0,\beta)}(t) =\\ (2n + &\beta +1)[(2n + \beta+2)(2n +  \beta )t - \beta^2]P_n^{(0,\beta)}(t)\\
 &- 2n(n + \beta)(2n + \beta + 2)P_{n-1}^{(0,\beta)}(t)
\end{align*}
translates (using $n = m + 2k$ and $\beta = -2k$) into the recurrence (\ref{recur}) for the polynomials $\frac{(m + 2k)!}{m!}H_{m,k}(1 - 2t)$.    
\end{proof}

Finally, to prove the proposition, we simply plug in $m = r - k - 1$ into the previous lemma and simplify our above expression for $c(n)$.
\end{proof}

Recall that for any ideal class $\AA$, we have defined $$r_{\AA,\chi}(j) = \sum_{\substack{\a \in \AA \\ \a \subset \O \\ \N(\a) = j}} \chi(\a).$$  
Putting together Lemma \ref{theta}, Proposition \ref{jacobi}, and the manipulation of symbols in \cite[I.6.5]{Nek}, we obtain
\begin{align*}
a_m\left(\int_{\Z_p^\times} \lambda d\tilde \Psi_\AA \right) & = \frac{(-1)^{r - k - 1}}{\binom{2r - 2}{r - k - 1}}m^{r-k-1}\left(\frac{D}{-N}\right)\sum_{D = D_1D_2} \left(\frac{D_2}{N\a}\right)\chi(\D_1)^{-1}\\
&\times\sum_{\substack{j +nN = |D_1|m\\  (p,j) = 1}}\sum_{\substack{d | n\\ (p,d) =1}} r_{\AA\D_1^{-1},\chi}(j) \left(\frac{D_2}{-dN}\right)\left(\frac{D_1}{|D_2|n/d}\right)\\
& \times \lambda\left(\frac{m|D_1| - nN}{|D_1|d^2}\right)
\times H_{r - k -1,k}\left(1 - \frac{2nN}{m|D_1|}\right).
\end{align*}

\begin{lemma}
$$r_{\AA\D_1^{-1},\chi}(j) = \chi(\D_2)^{-1}r_{\AA,\chi}(j|D_2|).$$
\end{lemma}

\begin{proof}
Since $\D_1$ is 2-torsion in the class group, the left hand side equals $ r_{\AA\D_1,\chi}(j)$.  The lemma now follows from the definitions once one notes that $\b \mapsto \b\D_2$ is a bijection from integral ideals of norm $j$ in $\AA\D_1$ to integral ideals of norm $j|D_2|$ in $\AA\D$.    
\end{proof}

Using the lemma and also the change of variables employed in \cite{Nek}, we obtain our version of \cite[Proposition 6.6]{Nek}.

\begin{proposition}
If $p | m$, then 
\begin{align*}
a_m\left(\int_{\Z_p^\times} \lambda d\tilde \Psi_\AA \right) =& \frac{(-1)^{r - 1}}{\binom{2r - 2}{r -k - 1}}m^{r-k-1}\left(\frac{D}{-N}\right)|D|^{-k}\sum_{\substack{1 \leq n \leq \frac{m|D|}{N}\\  (p,n) = 1}}r_{\AA,\chi}(m|D| - nN) \\
&\times H_{r - k -1,k}\left(1 - \frac{2nN}{m|D|}\right)\sum_{d | n}\epsilon_\AA(n,d)  \lambda\left(\frac{m|D| - nN}{|D|}\cdot \frac{d^2}{n^2}\right).
\end{align*}
Here, $\e_\AA(n,d) = 0$ if $(d,n/d, |D|) > 1$, and otherwise 
\[\e_\AA(n,d) = \left(\frac{D_1}{d}\right)\left(\frac{D_2}{-nN/d}\right)\left(\frac{D_2}{\N(\AA)}\right),\]
where $(d,|D|) = |D_2|$ and $D = D_1D_2$.  

\end{proposition}

\begin{proof}
The proof is as in \cite{Nek}.  We have also used the fact that $\chi(\D) = D^k$ to get the extra factor of $|D|^{-k}$ and the correct sign (recall that $D$ is negative!).  
\end{proof}

\begin{corollary}\label{fourier}
If $\left(\frac{D}{N}\right) = 1$ and $p | m$, then 
\begin{align*}
a_m&\left(\int_{\Z_p^\times} \log_p d\tilde \Psi_\AA \right) = \\
&\frac{(-1)^r} {\binom{2r - 2}{r -k - 1}}m^{r-k-1}|D|^{-k}\sum_{\substack{1 \leq n \leq \frac{m|D|}{N}\\  (p,n) = 1}}r_{\AA,\chi}(m|D| - nN)\sigma_\AA(n)H_{r - k -1,k}\left(1 - \frac{2nN}{m|D|}\right),
\end{align*}
with $$\sigma_\AA(n) = \sum_{d | n} \epsilon_\AA(n,d) \log_p\left(\frac{n}{d^2}\right).$$ 
\end{corollary}

\begin{proof}
As in \cite{PR1}.
\end{proof}

\section{Generalized Heegner cycles}\label{cycles} 
In the previous section we computed Fourier coefficients of $p$-adic modular forms closely related to the derivative of $L_p(f,\chi)$ at the trivial character and in the cyclotomic direction.  We expect similar looking Fourier coefficients to appear as the sum of local heights of certain cycles, with the sum varying over the finite places of $H$ which are prime to $p$.

These cycles should come from the motive attached to $f \otimes \Theta_\chi$.  Since $\Theta_\chi$ has weight $2k + 1$, work of Deligne and Scholl provides a motive inside the cohomology of a Kuga-Sato variety which is the fiber product of $2k-1$ copies of the universal elliptic curve over $X_1(|D|)$.  We work with a closely related motive, which we describe now.  

We fix an elliptic curve $A/H$ with the following properties:
\begin{enumerate}
\item $\End_H(A) = \O_K$.
\item $A$ has good reduction at primes above $p$.  
\item $A$ is isogenous to each of its $\Gal(H/K)$-conjugates.
\item $A^\tau \cong A$, where $\tau$ is complex conjugation.  
\end{enumerate}

\begin{remark}
Since $D$ is odd, we may even choose such an $A$ with the added feature that $\psi_A^2$ is an unramified Hecke character of type (2,0) (see \cite{Roh}).  In that case, $\psi^{2k}_A$ differs from $\chi$ by a character of $\Gal(H/K)$, so this is a natural choice of $A$, given $\chi$. In general, $\psi_A^{2k}\chi^{-1}$ is a finite order Hecke character. 
\end{remark} 

We will use a two-dimensional submotive of $A^{2k}$ whose $\ell$-adic realizations are isomorphic to those of the Deligne-Scholl motive for $\Theta_{\psi_A^{2k}}$ (see \cite{BDP3}).    

From Property (3), $A$ is isogenous to $A^\sigma$ over $H$ for each $\sigma \in G := \Gal(H/K)$.  If $\sigma$ corresponds to an ideal class $[\a] \in \Pic(\O_K)$ via the Artin map, then one such isogeny $\phi_\a : A \to A^\sigma$ is given by $A \to A/A[\a]$, at least if $\a$ is integral.  A different choice of integral ideal $\a' \in [\a]$ gives an isomorphic elliptic curve over $H$, and the maps $\phi_\a$ and $\phi_{\a'}$ will differ by endomorphisms of $A$ and $A^\sigma$.  

As in the introduction, let $Y(N)/\Q$ be the modular curve parametrizing elliptic curves with full level $N$ structure, and let $\E \to Y(N)$ be the universal elliptic curve with level $N$ structure.  The canonical non-singular compactification of the $(2r-2)$-fold fiber product $$\E \times_{Y(N)} \cdots \times_{Y(N)} \E,$$ will be denoted by $W = W_{2r-2}$ \cite{Sch}; $W$ is a variety over $\Q$.  The map $W \to X(N)$ to the compactified modular curve has fibers (over non-cuspidal points) of the form $E^{2r-2}$, for some elliptic curve $E$.  We set
$$X = X_{r,N,k} =  W_H  \times   A^{2k},$$ where $W_H$ is the base change to $H$.  Recall the curve $X_0(N)/\Q$, the coarse moduli space of generalized elliptic curves with a cyclic subgroup of order $N$.  $X_0(N)$ is the quotient of $X(N)$ by the action of the standard Borel subgroup $B \subset \GL_2\left(\Z/N\Z\right)/\{\pm1\}$.    


The computations of the Fourier coefficients in the previous section suggest that we consider the following \textit{generalized Heegner cycle} on $X$.  Fix a Heegner point $y \in Y_0(N)(H)$ represented by a cyclic $N$-isogeny $A \to A'$, for some elliptic curve $A'/H$ with CM by $\O_K$.  Such an isogeny exists since each prime dividing $N$ splits in $K$.   Also let $\tilde y$ be a point of $Y(N)_H$ over $y$.  The fiber $E_{\tilde y}$ of the universal elliptic curve $\E \to Y(N)$ above the point $\tilde y$ is isomorphic to $A_F$, where $F \supset H$ is the residue field of $\tilde y$.  Let 
$$\Delta \subset E_{\tilde y} \times A_F \cong A_F \times A_F$$ 
be the diagonal, and we write $\Gamma_{\sqrt D} \subset E_{\tilde y} \times E_{\tilde y}$ for the graph of $\sqrt D \in \End(E_{\tilde y}) \cong \O_K$.  We define 
$$Y = \Gamma_{\sqrt D}^{r-1-k} \times  \Delta^{2k}  \subset X_{\tilde y} \cong A_F^{2r-2} \times A_F^{2k},$$ so that $Y \in \CH^{k+r}(X_F)$.  Here $X_{\tilde y}$ is the fiber of the natural projection $X \to X(N)$ above the point $\tilde y$.

Since $X$ is not defined over $\Q$, we need to find cycles to play the role of $\Gal(H/K)$-conjugates of $Y$.   For each $\sigma \in \Gal(H/K)$ we have a corresponding ideal class $\AA$.  For each integral ideal $\a \in \AA$, define the cycle $Y^\a$ as follows:
\begin{equation*}Y^\a = \Gamma_{\sqrt{D}}^{r-k-1} \times \left(\Gamma^t_{\phi_\a}\right)^{2k}  \subset \left(A_F^\a \times A_F^\a \right)^{r-k-1} \times \left(A_F^\a \times A_F\right)^{2k}  = X_{\tilde y^\sigma} \subset X_F.\end{equation*}
Here, $\Gamma^t_{\phi_\a}$ is the transpose of $\Gamma_{\phi_\a}$, the graph of $\phi_\a: A \to A^\a$.  The cycle $Y^\a \in \CH^{k+r}(X_F)$ is \textit{not} independent of the class of $\a$ in $\Pic(\O_K)$, but certain expressions involving $Y^\a$ \textit{will} be independent of the class of $\a$.  Note that $Y = Y^{\O_K}$.  


\subsection{Projectors}

Next we define a projector $\e \in \Corr^0(X, X)_K$ so that $\e Y^\a$ lies in the group $\CH^{r+k}(X_F)_{0,K}$ of homologically trivial $(r+k)$-cycles with coefficients in $K$.  Here, $\Corr^0(X,X)_K$ is the ring of degree 0 correspondences with coefficients in $K$.  For definitions and conventions concerning motives, correspondences, and projectors see \cite[\S2]{BDP3}.  

The projector is defined as $\epsilon = \epsilon_X = \epsilon_W\epsilon_\ell$.  Here,  $\epsilon_W$ is the pullback to $X$ of the Deligne-Scholl projector $\tilde\e_W \in \Q[\Aut(W)]$ which projects onto the subspace of $H^{2r-1}(W)$ coming from modular forms of weight $2r$ (see e.g.\ \cite[\S2]{BDP1}).  The second factor $\epsilon_\ell$ is the pullback to $X$ of the projector 

$$\epsilon_\ell = \left(\frac{\sqrt{D} + [\sqrt{D}]}{2\sqrt{D}}\right)^{\otimes \ell}  \circ \left(\frac{1 - [-1]}{2}\right)^{\otimes \ell} \in \Corr^0(A^\ell, A^\ell)_K,$$
denoted by the same symbol. 
On the $p$-adic realization of the motive $M_{A^\ell,K}$, $\epsilon_\ell$ projects onto the 1-dimensional $\Q_p$-subspace $V_{\p}^{\otimes 2k}A$ of 
$$\Sym^{2k}H_\et^1(\bar A,\Q_p)(k) \subset H_\et^{2k}(\bar A^{2k},\Q_p(k)).$$ Here, $\p$ is the prime of $K$ above $p$ which is determined by our chosen embedding $K \into \bar \Q_p$ and $V_{\p}A = \mathop{\varprojlim}_{n} A[\p^n] \otimes \Q_p$ is the $\p$-adic Tate module of $A$.  See Section \ref{ordsec} and \cite[\S1.2]{BDP3} for more details. 

We also make use of the projectors 
$$\bar\e_\ell = \left(\frac{\sqrt{D} - [\sqrt{D}]}{2\sqrt{D}}\right)^{\otimes \ell }\circ \left(\frac{1 - [-1]}{2}\right)^{\otimes \ell} \in \Corr^0(A^\ell, A^\ell)_K$$ and $\kappa_\ell = \e_\ell + \bar \e_\ell$.  The first projects onto $V_{\bar \p}A^{\otimes \ell}$ and the latter onto $V_\p A^{\otimes \ell} \oplus V_{\bar \p}A^{\otimes \ell}$.  Set $\bar \e  = \e_W \bar\e_\ell$ and $\e'= \e_W \kappa_\ell$.

\begin{remark}\label{descend}
For this remark, suppose that $\chi = \psi^\ell$, where $\psi$ is the $(1,0)$-Hecke character attached to $A$ by the theory of complex multiplication.  This means the $G_H$-action on $H^1(\bar A, \Q_p)(1)$ is given by the $(K \otimes \Q_p)^\times$-values Galois character $\psi_H = \psi \circ \Nm_{H/K}$.  If we write $\chi_H = \psi_H^\ell$, then the motive $M(\chi_H)$ over $H$ (with coefficients in $K$) from Section \ref{apps} is defined by the triple $(A^{2k}, \kappa_\ell, k)$.  

We explain how to descend this to a motive over $K$ with coefficients in $\Q(\chi)$ (this a modification of a construction from an earlier draft of \cite{BDP3}).  Let $e_K$ and $\bar e_K$ be the idempotents in $K \otimes K$ corresponding to the first and second projections $K \otimes K \cong K \times K \to K$.  For each $\sigma \in \Gal(H/K)$ choose an ideal $\a \subset \O_K$ corresponding to $\sigma$ under the Artin map and define 
\[ \Gamma(\sigma) := e_K \cdot (\phi_\a \times \cdots \times \phi_\a) \otimes \chi(\a)^{-1} \in \Hom\left(A^\ell, (A^\ell)^\sigma\right) \otimes_\Q \Q(\chi)\]        
\[ \bar \Gamma(\sigma) := \bar e_K \cdot (\phi_\a \times \cdots \times \phi_\a) \otimes \bar\chi(\a)^{-1} \in \Hom\left(A^\ell, (A^\ell)^\sigma\right) \otimes_\Q \Q(\chi).\]        
Since $\chi(\gamma \a) = \gamma^\ell \chi(\a)$ and $\phi_{\gamma \a} = \gamma \phi_\a$, these definitions are independent of the choice of $\a$.  Moreover,
\[\Gamma(\sigma \tau) = \Gamma(\sigma)^\tau \circ \Gamma(\tau)\] 
and similarly for $\bar\Gamma$.  We set
\[\Lambda(\sigma) = \kappa_\ell \circ (\Gamma(\sigma) +  \bar \Gamma(\sigma)) \circ \kappa_\ell^\sigma \in \Corr^0(A^\ell, (A^\sigma)^\ell)_\Q \otimes_\Q \Q(\chi).\]
Then the collection $\{\Lambda(\sigma)\}_\sigma$ gives descent data for the motive $M(\chi_H) \otimes \Q(\chi)$, hence determines a motive $M(\chi)$ over $K$ with coefficients in $\Q(\chi)$.  The $p$-adic realization of $M(\chi)$ is $\chi \oplus \bar\chi$ where $\chi$ is now thought of as a $\Q(\chi)\otimes \Q_p$-valued character of $G_K$.  
\end{remark}

Define the following sheaf on $X(N)$: $$\L = j_*\Sym^w(R^1f_*\Q_p) \otimes \kappa_\ell H^{2k}_\et(\bar A^{2k}, \Q_p(k)),$$ where $w = 2r- 2$, and $j: Y(N) \hookrightarrow X(N)$ and $f: \E \to Y(N)$ are the natural maps.  

From now on we drop the subscript `$\et$' from all cohomology groups and set $\bar Z = Z \times_{\Spec k}  \Spec \bar k$ for any variety defined over a field $k$.  We also use the notation $V_K = V \otimes K$, for any abelian group $V$. 

\begin{theorem}
There is a canonical isomorphism
$$H^1(\bar X(N), \L) \xrightarrow{\sim} \epsilon' H^{2r+2k-1}(\bar X,\Q_p)= \e' H^*(\bar X, \Q_p).$$
\end{theorem}  

\begin{proof}
See \cite[II.2.4]{Nek} and \cite[Prop.\ 2.4]{BDP1}.
\end{proof}

\begin{corollary}\label{homtriv}
The cycles $\epsilon Y^\a$ and $\bar\e Y^\a$ are homologically trivial on $X_F$, i.e. they lie in the domain of the $p$-adic Abel-Jacobi map
$$\Phi: \CH^{r+k}(X_F)_{0,K} \to H^1(F,H^{2r+2k-1}(\bar X, \Q_p(r+k))).$$ 
\end{corollary}
\begin{proof}
By the theorem, $\e' Y^\a$ is in the kernel of the map 
$$\CH^{r+k}(X_F)_K \to H^{2r+2k}(\bar X_F, \Q_p(r+k)),$$ i.e. it is homologically trivial.  Moreover, $\e = \e \e'$ and $\bar \e = \bar \e \e'$.  Since Abel-Jacobi maps commute with algebraic correspondences, it follows that $\e Y^\a$ and $\bar\e Y^\a$ are homologically trivial as well.    
\end{proof}

\subsection{Bloch-Kato Selmer groups}
Let $F$ be a finite extension of $\Q_\ell$ ($\ell$ a prime, possibly equal to $p$) and let $V$ be a continuous $p$-adic representation of $\Gal(\bar F/F)$.  Then there is a Bloch-Kato subgroup $H^1_f(F, V) \subset H^1(F,V)$, defined for example in \cite{BK} or \cite[1.12 and 2.1.4]{Nekhts}.  If $\ell \neq p$ (resp.\ $\ell = p$) and $V$ is unramified (resp.\ crystalline), then $H^1_f(F,V) = \Ext^1(\Q_p, V)$ in the category of unramified (resp.\ crystalline) representations of $\Gal(\bar F/F)$.  If instead $F$ is a number field, then $H^1_f(F,V)$ is defined to be the set of classes in $H^1(F,V)$ which restrict to classes in $H^1_f(F_v,V)$ for all finite primes $v$ of $F$.    

The Bloch-Kato Selmer group plays an important role in the general theory of $p$-adic heights of homologically trivial algebraic cycles on a smooth projective variety $X/F$ defined over a number field $F$.  Indeed, Nekov\'{a}\v{r}'s $p$-adic height pairing is only defined on $H^1_f(F, V)$, and not on the Chow group $\CH^j(X)_0$ of homologically trivial cycles of codimension $j$.  Here $V = H^{2j-1}(\bar X, \Q_p(j))$.  This is compatible with the Bloch-Kato conjecture \cite{BK}, which asserts (among other, much deeper statements) that the image of the Abel-Jacobi map 
\[\Phi: \CH^j(X)_0 \to H^1(F,V)\]
is contained in $H^1_f(F,V)$.  The next couple of results follow \cite[II.2]{Nek} and verify this aspect of the Bloch-Kato conjecture in our situation, allowing us to consider $p$-adic heights of generalized Heegner cycles.  We also give a more concrete description of the Abel-Jacobi images of generalized Heegner cycles in terms of local systems on the modular curve.    

Denote by $b(Y^\a)$ the cohomology class of $\epsilon(\bar Y^\a)$ in the fiber $\bar X_{\tilde y}$, so that $b(Y^\a)$ lies in 
$$\e' H^{2r+2k -2}\left(\bar X_{\tilde y^\sigma}, \Q_p(r+k -1)\right)^{G(\bar F/F)} \xrightarrow{\sim} H^0\left(\overline{\tilde y^\sigma}, \B\right)^{G(\bar F/F)},$$ where 
$$\B = \Sym^{2r-2}(R^1f_*\Q_p)(r-1)  \otimes \kappa_\ell H^{2k}\left(\bar A^{2k}, \Q_p(k)\right),$$
the sheaf on $Y(N)$.  The isomorphism above follows from proper base change, Lemma 1.8 of \cite{BDP1}, and the Kunneth formula.  Similarly, let $\bar b(Y^\a)$ be the class of $\bar \e \bar Y^\a$.  For the next proposition, let $j: Y(N) \to X(N)$ be the inclusion.  
  
\begin{theorem}\label{AJ}
Set $V = H^{2r+2k-1}(\bar X, \Q_p(r+k))$. 
\begin{enumerate}
\item $V$ is a crystalline representation of $\Gal(\bar H_v/H_v)$ for all $v | p$.
\item The Abel-Jacobi images $z^\a = \Phi(\epsilon Y^\a), \bar z^\a = \Phi(\bar\e Y^\a) \in H^1(F, V)$ lie in the subspace $H^1_f\left(F, V\right).$ 
\item The element $z^\a$, thought of as an extension of $p$-adic Galois representations, can be obtained as the pull back of 
$$ 0 \to H^1(\overline{X(N)}, j_*\B)(1) \to H^1(\overline{X(N)} - \overline{\tilde y^\sigma}, j_*\B)(1) \to H^0(\overline{\tilde y^\sigma}, \B) \to 0$$
by the map $\Q_p \to H^0\left(\overline{\tilde y^\sigma}, \B\right)$ sending 1 to $b(Y^\a)$, and similarly for $\bar z^\a$.  In particular, $z^\a$ and $\bar z^\a$ only depend on $b(Y^\a)$ and $\bar b(Y^\a)$ respectively.  

\end{enumerate}
\end{theorem}

\begin{proof}
(1) follows from Faltings' theorem \cite{Falt} and the fact that $X$ has good reduction at primes above $p$.  (2) is a general result due to Nekov\'{a}\v{r}, see \cite[Theorem 3.1]{Nek2}.  To apply the result one needs to know the purity conjecture for the monodromy filtration for $X$.  But this is known for $W$ and $A^\ell$, so it holds for $X$ as well \cite[3.2]{Nek2}.  We note that (2) is ultimately a local statement at each place $v$ of $H$, and for $v | p$,  the approach taken in the proof of Theorem \ref{crysmixed} below gives an alternate proof of this local statement.   Statement (3) can be proved exactly as in \cite[II.2.4]{Nek}.
\end{proof}

\begin{definition}
If $F/H$ is a field extension, then a \textit{Tate vector} is an element in $H^0(\bar y_0, \B)^{\Gal(\bar F/F)}$ for some $y_0 \in Y(N)(F)$.  A \textit{Tate cycle} is a formal finite sum of Tate vectors over $F$.  The group of Tate cycles is denoted $Z(Y(N), F)$.   
\end{definition}

Let $\pi: X(N) \to X_0(N) = X(N)/B$ be the quotient map, and as in \cite{Nek}, define $\e_B = (\#B)^{-1}\sum_{g \in B} g$, which acts on $X(N)$ and its cohomology.  Set $\Aa = (\pi_* \B)^B$, $a(Y^\a) = \e_Bb(Y^\a)$, and $\bar a(Y^\a) = \e_B \bar b(Y^\a)$.  We define the group $Z(Y_0(N), F)$ of Tate cycles on $Y_0(N)$ exactly as for $Y(N)$, but with $\B$ replaced by $\Aa$.  Let $j_0: Y_0(N) \to X_0(N)$ be the inclusion.  Note that $a(Y^\a)$ is an element of $Z(Y(N),H)$, not just $Z(Y(N),F)$.        

\begin{proposition}
The element $\Phi(\e_B\e Y^\a) \in H^1\left(H, H^1\left(\overline{X_0(N)}, (j_0)_*\Aa\right)\left(1\right)\right)$, thought of as an extension of $p$-adic Galois representations, can be obtained as the pull back of 
$$ 0 \to H^1\left(\overline{X_0(N)}, j_*\Aa\right)(1) \to H^1\left(\overline{X_0(N)} - \bar y^\sigma, j_*\Aa\right)(1) \to H^0(\bar y^\sigma, \Aa) \to 0$$
by the map $\Q_p \to H^0\left(\overline{y^\sigma}, \Aa\right)$ sending 1 to $a(Y^\a)$.  In particular, $\Phi(\e_B\epsilon Y^\a)$ only depends on $a(Y^\a)$.  Similarly, $\Phi(\e_B \bar \e Y^\a)$ depends only on $\bar a(Y^\a)$.  
\end{proposition}

In fact, for any field $F/H$ one can define a map $\Phi_T: Z(Y_0(N), F) \to H^1(F, H^1(\bar X_0(N), j_{0*}\Aa)(1))$, by pulling back the appropriate exact sequence as above.  We then have $\Phi(\e_B \e Y^\a) = \Phi_T(a(Y^\a))$ and $\Phi(\e_B \bar \e Y^\a) = \Phi_T(\bar a Y^\a)$.  For more detail, see \cite[II.2.6]{Nek}.    


\subsection{Hecke operators}
The Hecke operators on $W_{2r-2}$ from \cite{Nek} pull back to give Hecke operators $T_m$ on $X$.  The $T_m$ are correspondences on $X$; they act on Chow groups and cohomology groups and commute with Abel-Jacobi maps.  To describe the action of the Hecke algebra $\T$ on Tate vectors, we need to say what $T_m$ does to an element of $H^0(\bar y_0, \Aa)^{G(\bar F/F)}$ for an arbitrary point $y_0 \in X_0(N)(F)$, $F$ an extension of $H$.  Such an element is represented by a triple $(E,C,b)$ where $E$ is an elliptic curve, $C$ is a subgroup of order $N$, and   
$$b \in \Sym^w(H^1(\bar E, \Q_p))(r-1) \otimes \kappa_\ell \Sym^{2k} (H^1(\bar A, \Q_p))(k).$$
As the Hecke operators are defined via base change from those on $W_{2r-2}$, we have:  
$$T_m(E,C,b) = \sum_{\substack{\lambda: E \to E' \\ \deg(\lambda) = m}} (E', \lambda(C), (\lambda^w \times \id)_*(b)),$$ where we are using the map $\lambda^w \times \id : E^w \times A^\ell \to E'^w \times A^\ell$.  

Now set $V_{r,A,\ell} = \e_B\e' V = H^1(\overline{X_0(N)}, (j_0)_*\Aa)(1),$ a subrepresentation of $V$.   Then $z^\a := \Phi(\e_B\epsilon Y^\a)$ lands in the Bloch-Kato subspace $H^1_f(H,V_{r,A,\ell}) \subset H^1(H,V_{r,A,\ell})$, by Proposition \ref{AJ}.  For any newform $f \in S_{2r}(\Gamma_0(N))$, we let $V_{f,A,\ell}$ be the $f$-isotypic component of $V_{r,A, \ell}$ with respect to the action of $\T$.  
Consider the $f$-isotypic Abel-Jacobi map $$\Phi_f: \CH^{r+k}(X)_{0,K} \to H^1_f(H,V_{f,A,\ell}),$$ and set $z^\a_f = \Phi_f(\e_B \e Y^\a)$ and $\bar z_f^\a = \Phi_f(\e_B\bar\e Y^\a)$.  

As is shown in Section \ref{ordsec}, the $p$-adic representation $V_{f,A,\ell}$ is ordinary and satisfies $V_{f,A,\ell} \cong V_{f,A,\ell}^*(1)$.  The results of \cite{Nekhts} therefore give a symmetric pairing
$$\langle \, , \rangle_{\ell_K}: H_f^1(H, V_{f,A,\ell}) \times H_f^1(H,V_{f,A,\ell}) \to \Q_p(f),$$
depending on a choice of logarithm $\ell_K : \A_K^\times /K^\times \to \Q_p$ and the canonical splitting of the local Hodge filtrations at places $v$ of $H$ above $p$.  We will sometimes omit the dependence on $\ell_K$ in the notation for the heights if a choice has been fixed.  If $a,b \in Z(Y_0(N),F)$ are two Tate cycles, then we will write $\left\langle a,b \right\rangle_{\ell_K} $ for $\left\langle \Phi_T(a), \Phi_T(b)\right\rangle_{\ell_K}$.      

\subsection{Intersection theory}\label{intersect}

Here we collect some facts about generalized Heegner cycles and their corresponding cohomology classes.  We first recall the intersection theory on products of elliptic curves; see \cite[II.3]{Nek} for proofs.    

Let $E, E', E''$ be elliptic curves over an algbraically closed field $k$ of characteristic not $p$, and set 
$$H^i(Y) = H_\et^i(Y,\Q_p) = \left(\lim_n H^i_\et (Y, \Z/p^n\Z)\right) \otimes \Q_p$$ for any variety $Y/k$.  A pair $(\alpha, \beta)$ of isogenies $\alpha\in \Hom(E'', E)$ and $\beta \in \Hom(E'', E')$, determines a cycle $$\Gamma_{\alpha,\beta}  = (\alpha,\beta)_*(1) \in \CH^1(E \times E'),$$
where $(\alpha,\beta)_* : \CH^0(E'') \to \CH^1(E \times E')$ is the push forward.          
The image of $\Gamma_{\alpha,\beta}$ under the cycle class map $\CH^1(E \times E') \to H^2(E \times E')(1)$ will be denoted by $[\Gamma_{\alpha,\beta}]$.  Also let $X_{\alpha,\beta}$ be the projection of $[\Gamma_{\alpha,\beta}]$ to $H^1(E) \otimes H^1(E')(1)$, i.e.
$$X_{\alpha,\beta} =  [\Gamma_{\alpha,\beta}] - \deg(\alpha)h - \deg(\beta)v,$$
where $h$ is the horizontal class $[\Gamma_{1,0}]$ and $v$ is the vertical class $[\Gamma_{0,1}]$.  If $\alpha \in \Hom(E, E')$, we write $\Gamma_\alpha$ and $X_{\alpha}$ for $\Gamma_{1,\alpha}$ and $X_{1,\alpha}$, respectively.  If $\beta \in \Hom(E' ,E)$ we write $\Gamma_\beta^t$ and $X_\beta^t$ for $\Gamma_{\beta, 1}$ and $X_{\beta, 1}$, respectively.  Finally, let 
$$(\, ,\,): H^2(E \times E')(1) \times H^2(E \times E')(1) \to \Q_p,$$
 be the non-degenerate cup product pairing.  
\begin{proposition}\label{bilin}
With notation as above,
\begin{enumerate}
\item The map $$\Hom(E'', E) \times \Hom(E'', E') \to H^1(E) \otimes H^1(E')(1)$$ given by $(\alpha,\beta) \mapsto  X_{\alpha,\beta}$ is biadditive.  
\item The map $\Hom(E, E') \to H^1(E) \times H^1(E')(1)$ given by $\alpha \mapsto X_{\alpha}$ is an injective group homomorphism.  
\item If $E = E'$, then $X_{\alpha,\beta} = X_{\beta \hat \alpha}$ and $(X_\alpha, X_\beta) = -\Tr(\alpha\hat \beta)$ for all $\alpha, \beta \in \End(E)$.
\end{enumerate}
Here, $\Tr :\End(E) \to \Z$ is the map $\alpha \mapsto \alpha + \hat\alpha$.    
\end{proposition} 
It is convenient to think of $H^1(E)$ as $V_pE^* = \Hom(V_pE, \Q_p)$, where $V_pE = T_pE \otimes \Q_p$ is the $p$-adic Tate module.  The Weil pairing $$V_pE \times V_pE \to \Q_p(1)$$ gives identifications $V_pE^*(1) \cong V_pE$ and $\bigwedge^2V_pE \cong \Q_p(1)$.  We then have the following diagram of isomorphisms 

\[\begin{CD}
\left(V_pE \otimes V_pE\right)(-1) @>>>  \left(\Sym^2V_pE \oplus \bigwedge^2 V_pE\right)(-1) @>>> \Sym^2V_pE (-1)\oplus \Q_p\\
@VVV    @.       @V\delta VV  \\
V_pE^* \otimes V_pE @>>> \End(V_pE) @>>>  \End_0(V_pE) \oplus \Q_p
\end{CD}\]

One checks that $\delta$ identifies $\Sym^2 V_pE(-1)$ with the space $\End_0(V_pE)$ of traceless endomorphisms of $V_pE$.  Now suppose that $E$ has complex multiplication by $\O_K$ and that $p = \p \bar\p$ splits in $K$.  Then 
$$V_pE = V_\p E \oplus V_{\bar\p}E,$$ where $V_\p = \varprojlim E[\p^n] \otimes \Q_p$ and $V_{\bar\p} = \varprojlim E[\bar\p^n] \otimes \Q_p$.  Let $x^*$ and $y^*$ be a basis for $V_\p E$ and $V_{\bar\p}E$ respectively, and let $x,y$ be the dual basis of $H^1(E)$ arising from the Weil pairing.  Since the Weil pairing is non-degenerate, we may assume that $e(x^*,y^*) = 1 \in \Q_p$. 

If $\alpha \in \End(E)$, then the class $X_\alpha \in H^1(E) \otimes H^1(E)(1)$, when thought of as an element of $\End(V_pE)$ via the isomorphisms above, is simply the map $V\alpha : V_pE \to V_pE$ induced on Tate modules.  Thus, $X_1 = \lambda(x \otimes y - y\otimes x)$ for some $\lambda \in \Q_p$.  Recall that one can compute the intersection pairing on $H^1(E)^{\otimes 2}$ in terms of the cup product on $H^1(E)$: 
$$(a \otimes b, c \otimes d) = -(a \cup c)(b\cup d).$$ Since $(X_1,X_1) = -2$, we conclude that $\lambda = 1$.  Next we claim that      
\begin{equation}\label{eigen}
X_{\sqrt{D}} = \pm \sqrt D(x \otimes y + y \otimes x).
\end{equation}              
To prove this, it suffices to show that $V\sqrt D$ acts on $V_\p$ by $\sqrt D$ and on $V_{\bar \p}$ by $-\sqrt D$.  Indeed, under the identifications
\[H^1(E) \otimes H^1(E)(1) \cong V_pE^* \otimes V_pE^*(1) \cong V_pE^* \otimes V_pE \cong \End(V_pE),\]
$x \otimes y$ corresponds to the element $f \in \End(V_p)$ such that $f(ax^* + by^*) = ax^*$ whereas $y \otimes x$ corresponds to $g \in \End(V_p)$ such that $g(ax^* + by^*) = -by^*$.

To understand how $V\sqrt D$ acts on $V_\p$, write $\p^n = p^n\Z + \frac{b + \sqrt D}{2}\Z$ for some $b,c \in \Z$ such that $b^2 - 4p^nc = D$, which is possible because $p$ splits in $K$.  For $P \in E[\p^n]$, one has $(b + \sqrt D)(P)= 0$, so $\sqrt D(P) = -bP$.  Since $b \equiv \pm \sqrt D$ (mod $\p^n)$, it follows upon taking a limit that $(V\sqrt D)(x^*) = \pm \sqrt D x^*$.  Since we can write $\bar\p^n = p^n\Z + \frac{b - \sqrt D}{2}\Z$, we also have $(V\sqrt D)(y^*) = \mp \sqrt Dy^*$, and this proves the claim.  Hence 
$$X_\gamma = \gamma(x \otimes y) - \bar\gamma (y \otimes x) \in H^1(E) \otimes H^1(E)(1),$$ 
for all $\gamma \in \O_K \hookrightarrow \End(E)$.  

Finally, note that the projector $\e_1 \in \Corr^0(E,E)_K$ defined earlier acts on $H^1(E)$ as projection onto $V_\p$.   

\begin{proposition}
Let $\a \subset \O_K$ be an ideal and $\AA \in \Pic(\O_K)$ its ideal class.  Then the elements
$$z_{f,\chi}^\AA = \chi(\a)^{-1}z_f^\a \hspace{5mm} \mbox{and}\hspace{5mm} z_{f,\bar\chi}^\AA = \bar\chi(\a)^{-1}\bar z_f^\a$$
in $H^1_f(H, V_{f,A,\ell})_{\bar\Q_p}$ depend only on $\AA \in \Pic(\O_K)$.  
\end{proposition}

\begin{proof}
To prove the proposition for $z_{f,\chi}^\AA$, we wish to relate $z_f^\a$ to $z_f^{\a(\gamma)}$ for some $\gamma \in \O_K$ and some integral ideal $\a$.  The contribution to $z_f^\a$ from one of the ``generalized" components $\Gamma_{\phi_\a}^t \subset A^\a \times A$ is $\e X_{\phi_\a, 1}$, where $X_{\phi_\a,1} \in H^1(\bar A^\a, \Q_p) \otimes H^1(\bar A, \Q_p)$ is the class of
$$\Gamma_{\phi_\a}^t - \deg(\phi_\a)h - v \in \CH^1(A^\a \times A),$$
as above.  Let $x,y$ be a basis of $H^1(\bar A, \Q_p)$ such that 
$$X_{\gamma,1} = \bar\gamma(x \otimes y) - \gamma(y \otimes x) \in H^1(\bar A, \Q_p) \otimes H^1(\bar A, \Q_p),$$
for all $\gamma \in \O_K$.  Let $x_\a, y_\a$ be the basis of $H^1(\bar A^\a, \Q_p)$ corresponding to $x,y$ under the isomorphism $\phi_\a^*: H^1(\bar A^\a, \Q_p) \to H^1(\bar A, \Q_p).$   One checks that $$(\phi_\a \times \id)^*(X_{\phi_\a, 1}) = \deg(\phi_\a)X_{1,1}$$ and so $$X_{\phi_\a, 1} = \deg(\phi_\a)\left(x_\a \otimes y - y_\a \otimes x\right).$$
Similarly,  $$X_{\phi_{\a(\gamma)},1} = X_{\gamma\phi_\a, 1} = \deg(\phi_\a)\left(\bar\gamma (x_\a \otimes y) - \gamma(y_\a \otimes x)\right).$$
Since the projector $\e$ kills $y$, we find that $\e X_{\gamma\phi_\a,1} = \gamma \e X_{\phi_\a,1}$.  In the components which come purely from the Kuga-Sato variety $W_{2r-2}$, the two cycles $Y^\a$ and $Y^{\a(\gamma)}$ are identical -- they both have the form $\e \Gamma_{\sqrt D}^{r-k-1}$.  Taking the tensor product of the $\ell$ ``generalized" components and the $r-k-1$ Kuga-Sato components, we conclude that $$z_f^{\a(\gamma)} = \gamma^\ell z_f^\a,$$ as desired.  The proof for $z_{f,\bar\chi}^\AA$ is similar: since $\bar z_f^\a$ is defined using $\bar\e $ instead of $\e$, the extra factor of $\bar\gamma^\ell$ which pops out is accounted for by the factor $\bar\chi(\a)^{-1}$.  
\end{proof}

\begin{lemma}\label{invariance}
For any ideal classes $\AA, \BB,\CC \in \Pic(\O_K)$, we have
$$\left\langle z_{f,\chi}^\AA, z_{f,\bar\chi}^\BB \right\rangle = \left\langle z_{f,\chi}^{\AA\CC} , z_{f,\bar\chi}^{\BB\CC}\right\rangle$$ 
\end{lemma}

\begin{proof}
It suffices to prove $\left\langle z^{\id}_{f,\chi}, z_{f,\bar\chi}^\BB \right\rangle = \left\langle z_{f,\chi}^\AA, z_{f,\bar\chi}^{\BB\AA}\right\rangle$ for all $\AA,\BB \in \Pic(\O_K)$.  Equivalently, we must show
\begin{equation}\label{htfunc}
\Nm(\a)^\ell \left\langle z^{\O_K}_f , \bar z_f^\b \right\rangle = \left\langle z_f^\a, \bar z_f^{\b\a}\right\rangle,
\end{equation}
for all integral ideals $\a$ and $\b$. 
Let $\sigma \in \Gal(\bar K/ K)$ restrict to an element of $\Gal(H/K)$ which corresponds to $\a$ under the Artin map.  Consider the morphisms of Chow groups $$\sigma: \CH^*(\overline{W \times A^\ell})_K \to \CH^*(\overline{W \times (A^\sigma)^\ell})_K$$ and
$$\xi = (\id \times \phi_\a^\ell)^*: \CH^*(\overline{W \times (A^\sigma)^\ell})_K \to \CH^*(\overline{W \times A^\ell})_K.$$  After identifying $A^\sigma$ with $A^\a$, one checks that $(\xi \circ \sigma)(Y^\b) = Y^{\a\b}$.  Indeed, since $\a$ and $\b$ are integral, the graph of $\phi_\b^\sigma : A^\sigma \to (A^\b)^\sigma$ can be identified with the graph of the projection map $\phi: A/A[\a] \to A/A[\a\b]$ (first note the two isogenies have the same kernel and then use the main theorem of complex multiplication).  The latter is pulled back to $\Gamma_{\phi_{\a\b}}$ by $(\id \times \phi_\a)^*$.  It follows that $(\xi \circ \sigma)(Y^\b) = Y^{\a\b}$, and the identity therefore holds for the corresponding cohomology classes.  On cohomology, $\sigma$ and $\xi$ are isomorphisms, so (\ref{htfunc}) follows from the functoriality of $p$-adic heights \cite[Theorem 4.11]{Nekhts}.  We are using the fact that $\left(\hat\phi_\a^\ell\right)^*$ is adjoint to $ \left(\phi_\a^\ell\right)^*$ under the pairing given by Poincar\'e duality, and that $\deg \phi_\a = \Nm(\a)$.      
\end{proof}

The goal now is to compute $\langle z_{f,\chi}, z_{f,\bar \chi}\rangle$, where
$$z_{f,\chi} = \frac{1}{h}\sum_{\AA \in \Pic(\O_K)} z_{f,\chi}^\AA \hspace{5mm} \mbox{and} \hspace{5mm}z_{f,\bar \chi} = \frac{1}{h}\sum_{\AA \in \Pic(\O_K)} z_{f,\bar\chi}^\AA.$$  Here, we have extended the $p$-adic height $\bar\Q_p$-linearly.  

Let $\tau \in \Gal(H/\Q)$ be a lift of the generator of $\Gal(K/\Q)$.  As $A$ and $W$ are defined over $\R$, $\tau$ acts on $X = W \times A^\ell$ and its cohomology. 

\begin{lemma}\label{atkin}
Let $\n\subset \O_K$ be the ideal of norm $N$ corresponding to the Heegner point $y \in X_0(N)$, and let $(-1)^r\e_f$ be the sign of the functional equation for $L(f,s)$.  Then $$\tau(z_{f,\chi}^\AA) = (-1)^{r-k-1}\e_f\chi(\n)N^{-k} z_{f,\bar\chi}^{\AA^{-1}[\bar\n]}$$ and
$$\tau(z_{f,\bar\chi}^\AA) = (-1)^{r-k-1}\e_f\bar\chi(\n)N^{-k} z_{f,\chi}^{\AA^{-1}[\bar\n]}.$$
\end{lemma}

\begin{proof}
Let $W^0_j(N)$ be the Kuga-Sato variety over $X_0(N)$, i.e. the quotient of $W_{j}$ by the action of the Borel subgroup $B$.  Recall the map $W_N: W^0_j \to W^0_j$ which sends a point $P \in \bar E^j$ in the fiber above a diagram $\phi: E \to E/E[\n]$ to the point $\phi^j(P)$ in the fiber above the diagram $\hat\phi: E/E[\n] \to E/E[N]$.  Meanwhile, complex conjugation sends the Heegner point $A^\a \to A^\a/A^\a[\n]$ to the Heegner point $A^{\bar\a} \to A^{\bar\a}/A^{\bar\a}[\bar \n]$.  Thus on a generalized component of our cycle, we have
$$(W_N \times \id)^*(X_{\phi_{\bar\a\bar\n},1}) = NX_{\phi_{\bar\a},1} = N\tau(X_{\phi_\a,1}),$$ where these objects are thought of as Chow cycles on $X$ which are supported on the fiber of $X$ above $(\tilde y)^{\sigma\tau}$.  Since $\tau$ takes $V_\p A$ to $V_{\bar \p} A$, we even have
$$ (W_N \times \id)^*(\bar \e_1 X_{\phi_{\bar\a\bar\n},1}) = N\bar\e_1X_{\phi_{\bar\a},1} = N\tau(\e_1X_{\phi_\a,1}).$$
On the purely Kuga-Sato components,  one computes \cite[6.2]{NekEuler} $$W_N ^*(X_{\sqrt D}) = NX_{\sqrt D} = -N\tau (X_{\sqrt D}),$$ where the $X_{\sqrt D}$ in the equation above are supported on $\tilde y^{\Frob(\bar \a \bar \n)}$, $\tilde y^{\Frob(\bar \a)}$, and $\tilde y^{\Frob(\a)}$ respectively.  

On the other hand, $(W_N \times \id)^2 = [N] \times \id$, where $[N]: W_{2r-2}^0 \to W_{2r-2}^0$ is multiplication by $N$ in each fiber.  On cycles and cohomology, $[N] \times \id$ acts as multiplication by $N^{2r-2}$.  Since $W_N$ commutes with the Hecke operators, we see that $(W_N \times \id)$ acts as multiplication by $\pm N^{r-1}$ on the $f$-isotypic part of cohomology, and this sign is well known to equal $\e_f$.  Putting things together, we obtain
$$\tau(z_f^\a) = \frac{(-1)^{r-k-1}(W_N \times \id)^*(\bar z_f^{\bar\a\bar\n})}{N^{2k+ r-k-1}}  = \frac{(-1)^{r-k-1}\e_f\bar z_f^{\bar\a\bar\n}}{N^{k}},$$
from which the first identity in the lemma follows.  The proof of the second identity is entirely analogous.  
\end{proof}

\begin{theorem}\label{Htvanish}
If $\ell_K : \A^\times_K/K^\times \to \Q_p$ is anticyclotomic, i.e. $\ell_K \circ \tau|_K = -\ell_K$, then 
$$\langle z_{f,\chi}, z_{f,\bar \chi}\rangle_{\ell_K} = 0.$$  In particular, Theorem \ref{main} holds for such $\ell_K$.  
\end{theorem}

\begin{proof}
From the previous lemma we have
$$\tau(z_{f,\chi}) = (-1)^{r-k-1}\e_f\chi(\n)N^{-k}z_{f,\bar\chi}$$ and
$$\tau(z_{f,\bar\chi}) = (-1)^{r-k-1}\e_f\bar\chi(\n)N^{-k}z_{f,\chi}.$$ Thus
$$\left\langle z_{f,\chi},z_{f,\bar\chi}\right\rangle_{\ell_K} = \left\langle \tau(z_{f,\chi}), \tau(z_{f,\bar\chi})\right\rangle_{\ell_K\circ \tau}= \left\langle z_{f,\bar\chi}, z_{f,\chi}\right\rangle_{-\ell_K} = -\left\langle  z_{f,\chi}, z_{f,\bar\chi}\right\rangle_{\ell_K},$$
which proves the vanishing.  Theorem \ref{main} now follows from Corollary \ref{Lvanish}.  
\end{proof}

Since any logarithm $\ell_K$ can be decomposed into a sum of a cyclotomic and an anticyclotomic logarithm, it now suffices to prove Theorem \ref{main} for cyclotomic $\ell_K$, i.e. we may assume $\ell_K = \ell_K \circ \tau|_K$.  By Lemma \ref{invariance} we have   
\begin{equation}\label{ortho}
\left\langle z_{f,\chi}, z_{f,\bar\chi}\right\rangle =  \frac{1}{h}\left\langle z_{f,\chi}^{\O_K}, z_{f,\bar\chi}\right\rangle=\frac{1}{h}\sum_{\AA \in \Pic(\O_K)} \left\langle z_f, z_{f,\bar\chi}^\AA\right\rangle.
\end{equation}
The height $\langle \, , \rangle$ can be written as a sum of local heights:
$$\langle x , y\rangle = \sum_v \langle x, y\rangle_v,$$
where $v$ varies over the \textit{finite} places of $H$. 
These local heights are defined in general in \cite{Nekhts} and computed explicitly for cyclotomic $\ell_K$ in \cite[Proposition II.2.16]{Nek} in a situation similar to ours.  In the next section we compute the local heights $\langle z_f, z_{f,\bar\chi}^\AA\rangle_v$ for finite places $v$ of $H$ not dividing $p$.  The contribution from local heights at places $v | p$ will be treated separately.  

\section{Local $p$-adic heights at primes away from $p$}\label{localhts}
Our goal is to compute $\left\langle z_f, z^\AA_{f,\bar\chi}\right\rangle_{\ell_K}$ when $\ell_K$ is cyclotomic.  Since such a homomorphism is unique up to scaling, we may assume that $\ell_K = \log_p \circ \lambda$, where $\lambda: G(K_\infty/K) \to 1 + p\Z_p$ is the cyclotomic character and $\log_p$ is Iwasawa's $p$-adic logarithm.  We may write $\lambda = \tilde \lambda \circ \textbf{N}$, where $\tilde \lambda : \Z_p^\times \to 1 + p\Z_p$ is given by $\tilde \lambda(x) = \langle x \rangle^{-1}$.  Here,  $\langle x \rangle  = x \omega^{-1}(x)$, where $\omega$ is the Teichmuller character.    

We maintain the following notations and assumptions for the rest of this section.  Fix an ideal class $\AA$ and an integer $m \geq 1$, and suppose that there are no integral ideals in $\AA$ of norm $m$, i.e. $r_\AA(m) = 0$.  Choose an integral representative $\a \in \AA$ and let $\sigma \in \Gal(H/K)$ correspond to $\AA$ under the Artin map.  Write $x = b(Y)$ and $\bar x^\a = \bar b (Y^\a)$ for the two Tate vectors supported at the points $y$ and $y^\sigma$ in $X_0(N)(H)$.  Let $v$ be a finite place of $H$ not dividing $p$ and set $F = H_v$.  Write $\Lambda$ for the ring of integers in $F^\ur$, the maximal unramified extension of $F$, and let $\F = \bar \F_\ell$ be the residue field of $\Lambda$.   Write $\underline X_0(N)\to \Spec \Z$ for the integral model of $X_0(N)$ constructed in \cite{KM}, and let $\underline X_0(N)_\Lambda$ be the base change to $\Spec \Lambda$. Finally, write $i: Y_0(N) \times_\Q F^\ur \hookrightarrow \underline X_0(N)_\Lambda$ for the inclusion.  

Now suppose $a, b$ are elements of $Z(Y_0(N),F^\ur)$ supported at points $y_a \neq y_b$ of $X_0(N)(F^\ur)$ of good reduction.  Let $\underline y_a$ and $\underline y_b$ be the Zariski closure of the points $y_a$ and $y_b$ in $\underline X_0(N)_\Lambda$ and let $\underline a$ and $\underline b$ be extensions of $a$ and $b$ to $H^0(\underline y_a, i_*\Aa)$ and $H^0(\underline y_b, i_*\Aa)$ respectively.  If $\underline y_a$ and $\underline y_b$ have common special fiber $z$ (so $z$ corresponds to an elliptic curve $E/\bar \F$), then define
$$(a,b)_v =   (\underline y_a \cdot \underline y_b)_z \cdot (\underline a_z, \underline b_z),$$
where $(\underline y_a \cdot \underline y_b)_z$ is the usual local intersection number on the arithmetic surface $\underline X_0(N)_\Lambda$ and $(\underline a_z, \underline b_z)$ is the intersection pairing on the cohomology of  $E^{2r-2} \times A_\F^\ell$, where $A_\F$ is the reduction of $A_{\bar F}$.  

\begin{remark}
Note that while $A$ may not have good reduction at $v$, it has potential good reduction.  We can therefore identify $H^i_\et(A_{\bar F}, \Q_p)$ and $H^i_\et(A_\F, \Q_p)$ as vector spaces, but not as $\Gal(\bar F/F)$-representations.  Since the ensuing intersection theoretic computations can be performed over an algebraic closure, this is enough for our purposes.          
\end{remark}

Our assumption that $r_\AA(m) = 0$ implies that the Tate vectors $x$ and $T_m \bar x^\a$ have disjoint support.  By \cite{ST}, we may assume that they are supported at points of $\X_0(N)_\Lambda$ which are represented by elliptic curves with good reduction.  The following proposition gives a way to compute the local heights purely in terms of Tate vectors.  This technique of computing heights of cycles on higher dimensional motives coming from local systems on curves is the key to the entire computation.  The idea goes back to work of Deligne, Beilinson, Brylinski, and Scholl, among others.      
  
\begin{proposition}
With notation and assumptions as above, we have
\begin{equation}\label{localht}
\left\langle x, T_m \bar x^\a \right\rangle_v =  -\left(x, T_m \bar x^\a\right)_v\log_p(\N v),
\end{equation}
\end{proposition}

\begin{proof}
The proof is exactly as in \cite[II.2.16 and II.4.5]{Nek}.  In our case, one uses that $H^2(\underline X_0(N),i_*\Aa(1) ) = 0$.  This follows from the fact that if $\Aa' =  \left(\pi_*\Sym^{2r-2}(R^1f_*\Q_p)(r-1)\right)^B$, then $\Aa = \Aa' \otimes W$, where $W$ is a trivial two-dimensional local system, and $H^2(\underline X_0(N), i_*\Aa') = 0$ \cite[14.5.5.1]{KM}.
\end{proof}
%
Recall that over $\Lambda$, the sections $\underline y$ and $\underline y^\sigma$ correspond to cyclic isogenies of degree $N$.  We will confuse the two notions, so that the notation $\Hom_\Lambda(\underline y^\sigma, \underline y)$ makes sense.  See \cite{Nek} and \cite{BC} for details.      

\begin{proposition}\label{deform}
Suppose $v$ is a finite prime of $H$ not divisible by $p$.  If $m \geq 1$ is prime to $N$ and satisfies $r_\AA(m) = 0$, then
$$(x,T_m \bar x^\a)_v = \frac{1}{2} m^{r-k - 1}\sum_{n \geq 1} \sum_g \left(\bar\e\left(X_{g \sqrt D g^{-1}}^{\otimes r - k - 1} \otimes X_{\overline{g\phi_\a}}^{\otimes \ell}\right), \epsilon\left(X_{\sqrt D }^{\otimes r - k - 1} \otimes X_1^{\otimes \ell}\right)\right),$$
where the sum is over $g \in \Hom_{\Lambda/\pi^n}(\underline y^\sigma, \underline y)$ of degree $m$.  The intersection pairing on the right takes place in the cohomology of $E^{2r-2} \times A_\F^\ell$, where $E \cong A_\F$ is the elliptic curve over $\F$ corresponding to the special fiber $\underline y_s$ of $\underline y$.  
\end{proposition} 

\begin{proof}
The proof builds on that of \cite[II.4.12]{Nek}, so we only mention what is new to our setting.  We write $m$ as $m = m_0q^t$ where $q$ is the rational prime below $v$ (this is what Nekov\'a\v{r} calls $\ell$).  In the notation of \cite{Nek}, we need to compute the special fiber of $\underline x_g^\a(j)$, where $g \in \Hom_\Lambda(\underline y^\sigma, \underline y_g^\sigma)$ is an isogeny of degree $m_0$.   There is no harm in assuming $r = k+1$, because the description of the purely Kuga-Sato components of $\underline x_g^\a(j)$ (i.e. coming from factors of the cycle $Y^\a $ of the form $\Gamma_{\sqrt D} \subset E^\a \times E^\a$) is handled in \cite{Nek}.

Assume now that $q$ is inert in $K$ and $t$ is even.  In this case the special fiber $(\underline y)_s$ is supersingular, and the special fiber $(\underline x_g^\a)_s$ of the Tate vector is represented by the pair
$$\left((\underline y_g^\sigma)_s, \bar\e\left(X^{\otimes \ell}_{g\phi_\a,1}\right)\right).$$
This follows from the definition of the Hecke operators and the following fact: if $g: E \to E'$ is an isogeny and $\phi: A \to E$ is an isogeny, then
$$\left(g \times \id\right)_*(\Gamma^t_\phi) = \Gamma^t_{g\phi} \in \CH^1(E' \times A).$$
Since any isogeny $h \in \Hom_{\Lambda/\pi^n}(\underline y^\sigma_g, \underline y)$ of degree $q^t$ on the special fiber $\underline y_s \cong (\underline y_g^\sigma)_s$ is of the form $q^{t/2}h_0$, with $h_0$ of degree 1, we find that, assuming $\underline y$ and $\underline y_g^\sigma(j)$ intersect, $(\underline x_g^\a(j))_s$ is represented by
\begin{align*}
\left((\underline y_g^\sigma)_s, \bar\e\left(X^{\otimes \ell}_{q^{t/2}g\phi_\a,1}\right)\right) &= \left(\underline y_s, \bar\e\left(X^{\otimes \ell}_{h_0q^{t/2}g\phi_\a,1}\right)\right)\\
 &= \left(\underline y_s, \bar\e\left(X^{\otimes \ell}_{hg\phi_\a,1}\right)\right) \\
 &= \left(\underline y_s, \bar\e\left(X^{\otimes \ell}_{\overline{hg\phi_\a}}\right)\right),
 \end{align*}
as desired.  The proof when $t$ is odd or when $q$ is ramified is similar.  If $q$ is split in $K$, then both sides of the equation are 0, as is shown in \cite{GZ}.   
\end{proof}


When $v$ lies over a non-split prime, $\End_{\Lambda/\pi}(\underline y) = \End(E)$ is an order $R$ in a quaternion algebra $B$ and we can make the double sum on the right hand side more explicit.  To do this, we follow \cite{GZ} and identify $\Hom_{\Lambda/\pi} (\underline y^\sigma, \underline y)$ with $R\a$ by sending a map $g$ to $b = g\phi_\a$.  The reduction of endomorphisms induces an embedding $K \hookrightarrow B$, which in turn determines a canonical decomposition $B = K \oplus Kj$.  Thus every $b\in B$ can be written as $b = \alpha + \beta j$ with $\alpha, \beta \in K$.  Recall also that the reduced norm on $B$ is additive with respect to this decomposition, i.e. $\N(b) = \N(\alpha) + \N(\beta j)$.     
\begin{proposition}
If $g\phi_\a = b = \alpha + \beta j \in \End(E)$, then 
\begin{align*}\left(\bar\e (X_{g\sqrt D g^{-1}}^{r-k-1} \otimes X_{\bar b}^{\otimes \ell}), \right. & \left. \e (X_{\sqrt D}^{\otimes r-k-1} \otimes X_1^{\otimes \ell})\right) =\\
& \frac{(4D)^{r-k-1}}{\binom{2r-2}{r-k-1}}\bar\alpha^{2k} H_{r-k-1,k}\left(1 - \frac{2\N(\beta j)}{\N(b)}\right),
\end{align*}
where 
\[H_{m,k}(t) = \frac{1}{2^m\cdot (m + 2k)!} \left(\frac{d}{dt}\right)^{m + 2k}[(t^2 - 1)^m(t - 1)^{2k}]\] 
\end{proposition}

\begin{proof}
Recall from Section \ref{intersect} that we have chosen a basis $x^*,y^*$ of $V_pE$, and a dual basis $x,y$ of $H^1(E)$ such that $x^* \in V_\p E$, $y^* \in V_{\bar \p} E$, and $(x^*,y^*) = 1$.  We have already seen that $X_\alpha = \alpha x \otimes y  - \bar\alpha y \otimes x$.  Since $\gamma  j = j \bar \gamma$ for all $\gamma \in K$, $Vj$ swaps $V_\p E$ and $V_{\bar \p}E$.  So we can write 
\[Vj =  \left( \begin{array}{cc}
0 & u \\
v & 0 \end{array} \right)
\]
for some $u,v \in \Q_p$ such that $uv = \N(j) = -j^2$.  It follows that 
\[X_b = \alpha x \otimes y - \bar \alpha y \otimes x + \beta u y \otimes y - \bar \beta v x \otimes x.\]

Next note that $g \sqrt D g^{-1} = b \sqrt D b^{-1}$.  We write $b \sqrt D b^{-1} = \gamma + \delta j$, so that $\gamma = \frac{\sqrt D}{\N(b)}(\N(\alpha) - \N(\beta j))$ and $\delta = \frac{-2\sqrt D}{\N(b)}\alpha\beta$.  Thus $X_{g \sqrt D g^{-1}}$ already lies in $\Sym^2 H^1(E)$, and hence (working now in the symmetric algebra) 
\[\bar \e X_{g\sqrt D g^{-1}} = 2\gamma xy + \delta uy^2 - \bar \delta v x^2 = \frac{2\sqrt D}{\N(b)}(\bar \alpha x - \beta uy)(\alpha y + \bar \beta vx),\]
since $\bar \e$ acts as Scholl's projector $\e_W$ on the purely Kuga-Sato components.  

The cohomology classes $X_{\bar b}$ in the statement of the proposition are on `mixed' components, i.e. they live in $H^1(E) \otimes H^1(E')$, where $E$ comes from a Kuga-Sato component and $E'$ (which is abstractly isomorphic to E) comes from the factor $A^\ell$.  Thus 
\[X_{\bar b} = \bar \alpha x \otimes y' - \alpha y \otimes x' - \beta u y \otimes y' + \bar\beta x \otimes x',\]
and $\bar \e X_{\bar b} = (\bar \alpha x - \beta u y)y'$, since $\bar \e$ acts trivially on $H^1(E)$ and kills the basis vector $x'$ in $H^1(E')$.  Using these observations together with the compatibility of the projectors with the multiplication in the appropriate symmetric algebras, we compute

\begin{align*}
\left(\bar\e \right. & \left. (X_{g\sqrt D g^{-1}}^{r-k-1} \otimes X_{\bar b}^{\otimes \ell}), \e (X_{\sqrt D}^{\otimes r-k-1} \otimes X_1^{\otimes \ell})\right)& \\ 
&= \left((2\gamma xy + \delta u y^2 - \bar \delta x^2)^{r-k-1}(\bar \alpha x - \beta u y)^{2k}\otimes y'^{2k}, (2\sqrt D xy)^{r-k-1}y^{2k}\otimes x'^{2k}\right)\\
&= \left(\frac{4D}{\N(b)}\right)^{r-k-1}(y'^{2k},x'^{2k})\left((\bar \alpha x - \beta uy)^{r+k-1}(\alpha y + \bar \beta vx)^{r-k-1}, x^{r-k-1}y^{r+k-1}\right)\\
&= \left(\frac{4D}{\N(b)}\right)^{r-k-1}(y'^{2k},x'^{2k})(y^{r-k-1}x^{r+k-1}, x^{r-k-1}y^{r+k-1})\cdot C\\
&= \frac{(4D)^{r-k-1}}{\N(b)^{r-k-1}\binom{2r-2}{r-k-1}}\cdot C,
\end{align*}        
where $C$ is the coefficient of the monomial $y^{r-k-1}x^{r+k-1}$ in $(\bar \alpha x - \beta uy)^{r+k-1}(\alpha y + \bar \beta vx)^{r-k-1}$.  The pairings in the second to last line are the natural ones on $\Sym^{2k}H^1(E')$ and $\Sym^{2r-2}H^1(E)$ induced from the pairings on the full tensor algebras.  For example, $\Sym^{2r-2} H^1(E)$ has a natural pairing coming from the cup product $(\, , \,)$ on $H^1(E)$: 
$$(v_1 \otimes \cdots \otimes v_{2r-2}) \times (w_1\otimes  \cdots \otimes w_{2r-2}) \mapsto  \frac{1}{(2r-2)!}\sum_{\sigma \in S_{2r-2}} \prod_{i =1}^{2r-2} (v_i, w_{\sigma(i)}).$$
In particular, $(x^ay^b, x^cy^d) = 0$ unless $a = d$ and $b =c$, and 
\[(x^ay^b, y^ax^b) = \frac{a!b!}{(a+b)!} = \binom{a+b}{a}^{-1}.\] We have also used that on $\Sym^{2r-2} H^1(E)\otimes \Sym^{2k} H^1(E')$ we have $\left(u \otimes v, w \otimes z\right) =  (u,w)(v,z)$.

To compute the value of $C$, note that in general, the coefficient of $x^{m + 2k}$ in 
\[(ax + b)^{m + 2k}(cx + d)^m\] is equal to $a^{2k}(ad - bc)^m H_{m,k}\left(\frac{ad + bc}{ad - bc}\right)$.  This is proved using the method of \cite[3.3.3]{Zhang}.  Applying this to the situation at hand, we find that
 \[C = \bar \alpha^{2k} \N(b)^{r-k-1} H_{r-k-1,k}\left(1 - \frac{2\N(\beta j)}{\N(b)}\right).\]  Plugging this in, we obtain the desired expression for the pairing on the special fiber.        
\end{proof}

For each prime $q$, define $\langle x, T_m \bar x^\a \rangle_q = \sum_{v | q} \langle x, T_m \bar x^\a \rangle_v.$

\begin{proposition}\label{htcoeff}
Assume that $(m,N) = 1$, $r_\AA(m) = 0$ and that $N > 1$.  Then 
\begin{align*}
\chi(\bar\a)&^{-1}\sum_{q \neq p} \langle x, T_m  \bar x^\a \rangle_q =\\
 &-u^2 \frac{\left(4|D|m\right)^{r-k-1}}{D^k\cdot \binom{2r-2}{r-k-1}} \sum_{0 < n < \frac{m|D|}{N}} \sigma_\AA(n)r_{\AA,\chi}\left(m|D| - nN\right)H_{r-k-1,k}\left(1-\frac{2nN}{m|D|}\right),
\end{align*}
with $\sigma_\AA(n)$ defined as in Corollary $\ref{fourier}$.
\end{proposition}

\begin{proof}
This type of sum arises from Proposition \ref{deform} exactly as in \cite[II.4.17]{Nek} and \cite{GZ}, so we omit the details.  The main new feature here is that each $b =  \alpha + \beta j \in R\a$ of degree $m$ is weighted by $\bar\alpha^\ell$, by the previous proposition.  Thus the numbers $r_\AA(j)$, with $j = m|D| - nN$, and which in \cite[II.4.17]{Nek} are simply counting the number of such $b$, become non-trivial sums of the form 
$$\sum_{\substack{\c \subset \O_K \\ [\c] = \AA^{-1}\D\\ \Nm(\c) = j}} \bar \alpha^\ell.$$  
Here, $\alpha \in \ \d^{-1}\a$ and $\c = (\alpha)\d \a^{-1}$ (see \cite[p. 265]{GZ}).  
Rewriting this sum, we obtain
\[\sum_{\substack{\c \subset \O_K \\ [\c] = \AA^{-1}\D\\ \Nm(\c) = j}} \bar \chi (\c \a\d^{-1}) =\frac{\chi(\bar\a)}{\chi(\d)}\cdot \sum_{\substack{\c \subset \O_K \\ [\c] = \AA^{-1}\D\\ \Nm(\c) = j}} \chi(\bar \c)
= \frac{\chi(\bar \a)}{D^k} \cdot \sum_{\substack{\c \subset \O_K \\ [\c] = \AA\\ \Nm(\c) = j}} \chi(\c) = \frac{\chi(\bar \a)}{D^k} r_{\AA, \chi}(j).\]
Multiplying by $\chi(\bar \a)^{-1}$, we get the desired result.
\end{proof}

We define
$$B_m^\sigma = m^{r-k-1}\sum_{\substack{n= 1\\ (p,n) =1}}^{\frac{m|D|}{N}} r_{\AA,\chi}(m|D|-nN)\sigma_\AA(n)H_{r-k-1,k}\left(1 - \frac{2nN}{m|D|}\right)$$
$$C_m^\sigma = m^{r-k-1}\sum_{n=1}^{\frac{m|D|}{N}} r_{\AA,\chi}(m|D|-nN)\sigma_\AA(n)H_{r-k-1,k}\left(1 - \frac{2nN}{m|D|}\right)$$ 
Up to a constant, the $B_m^\sigma$ appear as coefficients of the derivative of the $p$-adic $L$-function defined earlier and $C_m^\sigma$ contributes to the height of our generalized Heegner cycle.  Just as in \cite[I.6.7]{Nek}, we wish to relate the $B_m^\sigma$ to the $C_m^\sigma$.

Let $U_p$ be the operator defined by $C^\sigma_m \mapsto C^\sigma_{mp}$ and similarly for $B_m^\sigma$.  For a prime $\p$ of $K$ above $p$, we write $\sigma_\p$ for $\Frob(\p) \in \Gal(H/K)$.    We will also let $\sigma_\p$ be the operator $C_m^\sigma \mapsto C_m^{\sigma \sigma_\p}$.  

\begin{proposition}\label{mainid}
Suppose $p>2$ is a prime which splits in $K$ and that $\chi$ is an unramified Hecke character of $K$ of infinity type $(\ell,0)$ with $\ell = 2k$.  Then
$$\prod_{\p | p} \left(U_p - p^{r-k -1}\chi(\bar\p)\sigma_\p\right)^2 C_m^\sigma = \left(U^4_p - p^{2r-2}U_p^2\right) B_m^\sigma.$$
\end{proposition}
\begin{proof}
The proof follows \cite[Proposition 3.20]{PR1}, which is the case $r = 1$ and $\ell = k= 0$.  We first generalize \cite[Lemma 3.11]{PR1} and write down relations between the various $r_{\AA,\chi}(-)$.   

\begin{lemma}
Set $r_{\AA,\chi}(t) = 0$ if $t \in \Q \setminus \NN$. For all integers $m > 0$, we have  
\begin{enumerate}
\item $r_{\AA, \chi}(mp) + p^\ell r_{\AA,\chi}(m/p) = \chi(\bar \p)r_{\AA\p,\chi}(m) + \chi(\p)r_{\AA\bar \p, \chi}(m)$.
\item $r_{\AA,\chi}(mp^2) + p^{2\ell}r_{\AA,\chi}(m/p^2) = \chi(\bar \p^2)r_{\AA\p^2,\chi}(m) + \chi(\p^2)r_{\AA\bar\p^2,\chi}(m)$ if $p | m$.
\item $r_{\AA,\chi}(mp^2) - p^\ell r_{\AA,\chi}(m) = \chi(\bar \p^2)r_{\AA\p^2,\chi}(m) +  \chi(\p^2)r_{\AA\bar \p^2,\chi}(m)$ if $p \nmid m$.
\item If $n = n_0p^t$ with $p \not \divides n_0$, then $\sigma_\AA(n) = (t+1)\sigma_{\AA,t}(n_0)$, where $\sigma_{\AA,t} = \sigma_{\AA\p^t} = \sigma_{\AA\bar\p^t}.$
\item $\sigma_{\AA\b^2}(n) = \sigma_\AA(n)$ for any ideal $\b$.  
\end{enumerate}
\end{lemma}

\begin{proof}
Note that every integral ideal $\a$ in $\AA$ of norm $mp$ is either of the form $\a' \p$ with $\a' \in \AA\bar \p$ of norm $m$ or it is of the form $\a' \bar \p$ with $\a' \in \AA\p$ of norm $m$.  Moreover, an ideal of norm $mp$ which can be written as such a product in two ways is necessarily the product of an integral ideal in $\AA$ of norm $m/p$ with $(p)$.  The first claim now follows from the fact that 
$$r_{\AA,\chi}(t) = \sum_{\substack{\a \subset \O \\ \a \in \AA\\ \N(\a) = t}} \chi(\a),$$ and that $\chi((p)) = p^\ell$.  Parts (2) and (3) follow formally from (1).  (4) is proven in \cite{PR1} and (5) is clear from the definition.    
\end{proof}
Going back to the proof of Proposition \ref{mainid}, the LHS is equal to
\begin{align*}
C_{mp^4}^\s - 2p^{r-k-1}&\left(\chi(\bar \p) C_{mp^3}^{\s \s_\p}+ \chi(\p) C_{mp^3}^{\s \s_{\bar \p}}\right) \\
&+p^{2(r-k-1)}\left(\chi(\bar\p)^2 C_{mp^2}^{\s \s_{\p^2}} + 4p^\ell C_{mp^2}^\s + \chi(\p)C_{mp^2}^{\s \s_{\bar \p^2}}\right)\\
&-2p^{3(r-k-1)+\ell} \left(\chi(\bar\p) C_{mp}^{\s \s_\p} + \chi(\p) C_{mp}^{\s \s_{\bar \p}}\right) + p^{4(r-1)}C_m^\s.
\end{align*}
In the following we write $v(p)$ for the $p$-adic valuation of an integer $n$, and  $n = n_0p^{v(p)}$.  For the sake of brevity we also set $r_\AA(u,v) = r_{\AA,\chi}(u|D| - vN)$ for integers $u$ and $v$ and $H(x) = H_{r-k-1,k}(x)$.  Then by the lemma, the LHS above is equal to 
$$\sum_{n = 1}^{m|D|/N}(v(n)+1)(mp^4)^{r-k-1} M(n),$$ where $M(n)$ equals
\begin{align*}
r_\AA&(mp^4, n)\sigma_{\AA,v(n)}(n_0)H\left(1 - \frac{2nN}{mp^4|D|}\right)\\
&-2\left[r_\AA(mp^4, pn) + p^\ell r_\AA\left(mp^2, n/p\right)\right]\sigma_{\AA, v(n) + 1}(n_0)H\left(1 - \frac{2nN}{mp^3|D|}\right)\\
&+ \left[r_\AA(mp^4, p^2n) + \begin{cases} p^{2\ell}r_\AA\left(m, n/p^2\right) + 4p^\ell r_\AA(mp^2, n) &\mbox{if } p | n \\ 3p^\ell r_\AA(mp^2, n) & \mbox{if } p \not \divides n \end{cases} \right]\\
& \hspace{25mm}\times \sigma_{\AA,v(n)}(n_0)H\left(1 - \frac{2nN}{mp^2|D|}\right)\\
&-2p^\ell\left[r_\AA(mp^2, pn) + p^\ell r_\AA(m, n/p)\right]\sigma_{\AA,v(n)+1}(n_0)H\left(1 - \frac{2nN}{mp|D|}\right)\\
&+p^{2\ell} r_\AA(m,n)\sigma_{\AA,v(n)}(n_0)H\left(1 - \frac{2nN}{m|D|}\right).
\end{align*}
Grouping in terms of the $n_0$ which arise in this sum, we find that the LHS is equal to
$$\sum_{(n_0,p) = 1} \sum_t \sigma_{\AA, t}(n_0)A_t$$ where $A_t$ equals
\begin{align*}
(mp^4)&^{r-k-1}r_\AA(mp^4, p^tn_0)\left[t + 1 - 2t + \begin{cases} t - 1 &\mbox{if } t \geq 1 \\ 0 & \mbox{if } t = 0 \end{cases} \right]H\left(1 - \frac{2n_0p^tN}{mp^4|D|}\right)\\
&+ (mp^2)^{r-k-1}p^{2r-2}r_\AA(mp^2,p^tn_0)\left[-2(t+2) + \begin{cases} 4(t+1) - 2t &\mbox{if } t \geq 1 \\ 3 & \mbox{if } t = 0 \end{cases}\right]\\
& \hspace{25mm} \times H\left(1 - \frac{2n_0p^tN}{mp^2|D|}\right)\\
&+ m^{r-k-1}p^{4r-4}r_\AA(m, p^tn_0)\left[t+3 - 2(t+2) + t + 1\right]H\left(1 - \frac{2n_0p^tN}{m|D|}\right).
\end{align*}
So $A_t = 0$ unless $t = 0$, and we conclude that the LHS is equal to $(U_p^4 - p^{2r-2}U_p^2)B_m^\s$, as desired.  
\end{proof}

\section{Ordinary representations}\label{ordsec} 
The contributions to the $p$-adic height $\langle z_f, z_{f,\bar\chi}^\AA \rangle$ coming from places $v | p$ will eventually be shown to vanish.  The proof is as in \cite{Nek} (though see Section \ref{nekfix}), where the key fact is that the local $p$-adic Galois representation $V_f$ attached to $f$ is ordinary.  We recall this notion and prove that the Galois representation $V_{f,A,\ell} = V_f \otimes \kappa_\ell H^\ell(\bar A^\ell,\Q_p)(k)$ is ordinary as well.   

\begin{definition}
Let $F$ be a finite extension of $\Q_p$.  A $p$-adic Galois representation $V$ of $G_F = \Gal(\bar F/ F)$ is \textit{ordinary} if it admits a decreasing filtration by subrepresentations
$$ \cdots F^i V \supset F^{i+1} V  \supset \cdots$$ such that $\bigcup F^i V = V$, $\bigcap F^i V = 0,$ and
 for each $i$, $F^i V / F^{i+1}V = A_i(i)$, with $A_i$ unramified.   
\end{definition}

Recall we have defined $\e' = \e_W\kappa_\ell$ with
$$\kappa_\ell = \left[\left(\frac{\sqrt{D} + [\sqrt{D}]}{2\sqrt{D}}\right)^{\otimes \ell} + \left(\frac{\sqrt{D} - [\sqrt{D}]}{2\sqrt{D}}\right)^{\otimes \ell}    \right]\circ \left(\frac{1 - [-1]}{2}\right)^{\otimes \ell}.$$
 
\begin{theorem}\label{ord} 
Let $f \in S_{2r}(\Gamma_0(N))$ be an ordinary newform and let $V_f$ be the $2$-dimensional $p$-adic Galois representation associated to $f$ by Deligne.   Let $A/H$ be an elliptic curve with CM by $\O_K$ and assume $p$ splits in $K$ and $A$ has good reduction at primes above $p$.  For any $\ell = 2k \geq 0$, set $W = \kappa_\ell H^\ell(\bar A^\ell,\Q_p)(k)$. Then for any place $v$ of $H$ above $p$, $V_{f,A,\ell} = V_f \otimes W$ is an ordinary $p$-adic Galois representations of $\Gal(\bar H_v /  H_v)$.  
\end{theorem}

\begin{proof}
First we recall that $V_f$ is ordinary.  Indeed, Wiles \cite{Wi} proves that the action of the decomposition group $D_p$ on $V_f$ is given by 
\[ \left( \begin{array}{cc}
\epsilon_1 & * \\
0 & \epsilon_2 \end{array} \right)
\]
with $\epsilon_2$ unramified.  Since, $\det V_f$ is $\chi_\cyc^{2r-1}$, we have $\epsilon_1 = \epsilon_2^{-1}\chi^{2r-1}_\cyc$.  Thus, the filtration 
$$F^0V_f = V_f \supset F^1V_f  = F^{2r-1} V_f = \epsilon_1 \supset F^{2r}V_f = 0,$$ shows that $V_f$ is an ordinary $\Gal(\bar \Q_p/\Q_p)$-representation and hence an ordinary $\Gal(\bar H_v/H_v)$-representation as well.  Next we describe the ordinary filtration on (a Tate twist of) $W$.

\begin{proposition}
Write $(p) = \p \bar \p$ as ideals in $K$.  Then the $p$-adic representation $M  = \kappa_\ell H^\ell_\et(\bar A^\ell, \Q_p)(\ell)$ of $\Gal(\bar H_v/ H_v)$ has an ordinary filtration 
$$F^0M = M \supset F^1M = F^\ell M \supset F^{\ell+ 1}M =  0.$$  
\end{proposition}

\begin{proof}
The theory of complex multiplication associates to $A$ an algebraic Hecke character $\psi: \A^\times_H \to K^\times$ of type $\Nm: H^\times \to K^\times$ such that for any uniformizer $\pi_v$ at a place $v$ not dividing $p$ or the conductor of $A$, $\psi(\pi_v) \in K \cong \End(A)$ is a lift of the Frobenius morphism of the reduction $A_v$ at $v$.  The composition 
$$t_p: \A^\times_H \stackrel{\Nm}\longrightarrow \A^\times_K \to (K \otimes \Q_p)^\times$$ agrees with $\psi$ on $H^\times$, giving a continuous map
$$\rho' = \psi t_p^{-1}: \A^\times_H/H^\times \to (K \otimes \Q_p)^\times.$$
Since the target is totally disconnected, this factors through a map
$$\rho: G_H^\ab \to (K\otimes \Q_p)^\times.$$  By construction of the Hecke character (and the Chebotarev density theorem), the action of $\Gal(\bar H/H)$ on the rank 1 $(K \otimes \Q_p)$-module $T_p A \otimes \Q_p$ is given by the character $\rho$.  Since $p$ splits in $K$, we have $$(K \otimes \Q_p)^\times \cong K_\p^\times \oplus K_{\bar \p}^\times = \Q_p^\times \oplus \Q_p^\times.$$  

Now write $\rho = \rho_\p \oplus \rho_{\bar \p}$, where $\rho_\p$ and $\rho_{\bar \p}$ are the characters obtained by projecting $\rho$ onto $K_\p^\times$ and $K_{\bar\p}^\times$.  

\begin{lemma}
Let $\chi_\cyc: \Gal(\bar H_v / H_v) \to \Q_p^\times$ denote the cyclotomic character and consider $\rho_\p$ and $ \rho_{\bar \p}$ as representations of $\Gal(\bar H_v/H_v)$.  Then
$\rho_\p\rho_{\bar \p} = \chi_\cyc$ and $\rho_{\bar \p}$ is unramified.  
\end{lemma}

\begin{proof}
The non-degeneracy of the Weil pairing shows that $\bigwedge^2 T_p A \cong \Z_p(1)$.  It then follows from the previous discussion that $\rho_\p\rho_{\bar \p} = \chi_\cyc$.  That $\rho_{\bar \p}$ is unramified follows from the fact that $t_{\bar \p}(H_v) = 1$ and $v$ is prime to the conductor of $\psi$.  Indeed, the conductor of $A$ is the square of the conductor of $\psi$ \cite{Gr}, and $A$ has good reduction at $p$.  
\end{proof}

\begin{remark}
Let $\mathcal{A}/\O_H$ be the N\'eron model of $A/H$. Since $\mathcal{A}[\bar\p^n]$ is \'etale, it follows that the $\bar\p$-adic Tate module $V_{\bar \p}A$ is unramified at $v$.  We can therefore identify $\rho_\p \cong V_\p A$ and $\rho_{\bar\p} = V_{\bar \p}A$.  One can also see this from the computation in equation \ref{eigen}.  
\end{remark}

\begin{lemma}
As $\Gal(\bar H_v/ H_v)$-representations, $$H^1_\et(\bar A, \Q_p)(1) \cong \rho_\p \oplus \rho_{\bar\p}$$ and
$$M = \kappa_\ell H^\ell_\et(\bar A^\ell, \Q_p)(\ell) \cong \rho_\p^\ell \oplus \rho_{\bar\p}^\ell.$$ 
\end{lemma}

\begin{proof}
The first claim follows from the fact that \[T_pA \otimes \Q_p \cong H^1_\et(\bar A, \Q_p)(1).\]  Fix an embedding $\iota : \End(A) \into K$, which by our choices, induces an embedding $\End(A) \into \Q_p$.  By the definition of $\rho$, $\rho_\p$ is the subspace of $H^1_\et(\bar A, \Q_p)(1)$ on which $\alpha \in \End(A)$ acts by $\iota(\alpha)$, whereas on $\rho_{\bar \p}$, $\alpha$ acts as $\bar \iota(\alpha)$.  The second statement now follows from the Kunneth formula and the definition of $\kappa_\ell$.   
\end{proof}

Now set $F^0M = M$, $F^1M = F^\ell M = \psi^\ell$, and $F^{\ell + 1}M = 0$.  By the lemmas above, this gives an ordinary filtration of $M$ and proves the proposition.   
\end{proof}

Now to prove the theorem.  We have specified ordinary filtrations $F^iV_f$ and $F^iM$ above.  A simple check shows that 
$$F^i(V_f \otimes M) = \sum_{p + q = i} F^pV_f \otimes F^qM $$
is an ordinary filtration on $V_f \otimes M $.  Since $V_{f,A,\ell} = V_f \otimes W= (V_f \otimes M)(-k)$ and Tate twisting preserves ordinarity, this proves $V_{f,A,\ell}$ is ordinary. 
\end{proof}

\begin{remark}
Another way to obtain the ordinary filtration on $M$ is to use the fact that $M$ is isomorphic to the $p$-adic realization of the motive $M_{\theta_{\psi^\ell}}$ attached to the modular form $\theta_{\psi^\ell}$ of weight $\ell + 1$.  Since $A$ has ordinary reduction at $p$, $\theta_\psi$ is an ordinary modular form, and it follows that $\theta_{\psi^\ell}$ is ordinary as well.  We may therefore apply Wiles' theorem again to obtain an ordinary filtration on $W$.
\end{remark}

\begin{proposition}
The $\Gal(\bar H/H)$ representation $V_{f,A,\ell} = V_f \otimes W$ satisfies $V_{f,A,\ell}^*(1) \cong V_{f,A,\ell}$.  
\end{proposition}

\begin{proof}
Recall that $V_f^*(1) \cong V_f$, so we need to show that $W^* \cong W$.  This follows from the two lemmas above.  
\end{proof}

\section{Proof of Theorem \ref{main}}\label{proof} 
In what follows, normalized primitive forms $f_\beta  \in S_{2r}(\Gamma_0(N))$ will be indexed by the corresponding $\Q$-algebra homomorphisms $\beta: \T \to \bar \Q$.  We let $\beta_0$ be the homomorphism corresponding to our chosen newform $f$.  
If $\AA \in \Pic(\O_K)$, then  $$F_\AA := \sum_{\beta} \langle z_{\beta,\chi}, z_{\beta,\bar\chi}^\AA \rangle f_\beta$$
is a cusp form in $S_{2r}(\Gamma_0(N); \Q_p(\chi))$.  Indeed, for $(m,N) = 1$, we have
$$\chi(\bar\a)a_m(F_\AA) =\sum_{\beta} \langle z_\beta, \bar z_\beta^\a\rangle \beta(T_m) = \langle z, T_m \bar z^\a\rangle = \langle x, T_m \bar x^\a\rangle \in \Q_p,$$  
because the Hecke operators are self-adjoint with respect to the height pairing.  If $r_\AA(m) = 0$, then we have the decomposition
$$a_m(F_\AA) = c_m^\sigma + d_m^\sigma$$
where
$$c_m^\sigma = \chi(\bar\a)^{-1}\sum_{v \notdivides p} \langle x, T_m \bar x^\a\rangle_v, \hspace{5mm} d^\sigma_m = \chi(\bar\a)^{-1} \sum_{v | p} \langle x, T_m \bar x^\a \rangle_v,$$
and the sums are over \textit{finite} places of $H$.  

Both sides of the equation in Theorem \ref{main} depend linearly on a choice of arithmetic logarithm $\ell_K : \A^\times_K/K^\times \to \Q_p$.  By Theorem \ref{Htvanish}, it suffices to proves the main theorem for cyclotomic $\ell_K$, i.e. $\ell_K = \ell_K \circ \tau$.  As cyclotomic logarithms are unique up to scalar we only need to consider the case $\ell_K = \ell_\Q \circ  \textbf{N}$.  Thus, $\ell_K = \log_p \circ \lambda$, where $\lambda: G(K_\infty/K) \to 1 + p\Z_p$ is the cyclotomic character.  As before, we write $\lambda = \tilde \lambda \circ \textbf{N}$, where $\tilde \lambda : \Z_p^\times \to 1 + p\Z_p$ is given by $\tilde \lambda(x) = \langle x \rangle^{-1}$.

By definition, $$\left.L_p'(f \otimes \chi , \mathbbm{1}) = \frac{d}{ds} L_p(f \otimes \chi, \lambda^s)\right|_{s = 0}.$$
Also by definition,
\begin{align*} L_p(f \otimes \chi, \lambda^s) &= (-1)^{r-1}H_p(f) \left(\frac{D}{-N}\right)\left(1-C\left(\frac{D}{C}\right)\lambda^s(C)^{-1}\right)^{-1}\\
&\hspace{20mm} \times \int_{G(H_{p^\infty}(\mu_{p^\infty})/K)}\lambda^s d\tilde \Psi_{f,1,1}^C
\\&=(-1)^r H_p(f) \left(1-C\left(\frac{D}{C}\right)\tilde \lambda^{-2s}(C)\right)^{-1}\\
& \hspace{20mm} \times \int_{G(H_{p^\infty}(\mu_{p^\infty})/K)}\lambda^s d\tilde \Psi_{f,1,1}^C,
\end{align*}
where $C$ is an arbitrary integer prime to $N|D|p$.  The measure $\tilde \Psi^C_{f,1,1}$ is given by:
$$\tilde \Psi_{f,1,1}^C(\sigma (\mod p^n), \tau (\mod p^m)) = L_{f_0}(\tilde \Psi_{\AA,1}^C(a (\mod p^m)))$$ where $a$ corresponds to the restriction of $\tau$ under the Artin map and $\sigma$ corresponds to $[\AA] \in \Pic(\O_{p^n})$.  We have 
 \begin{align*}
 L_p&(f \otimes \chi)( \lambda^s) =\\
 & (-1)^r H_p(f) \left(1-C\left(\frac{D}{C}\right)\langle C\rangle ^{2s}\right)^{-1}L_{f_0}\left[\sum_{\AA \in \Pic(\O_K)} \int_{\Z_p^\times}\langle x \rangle^{-s} d\tilde \Psi_{\AA,1}^C\right].
 \end{align*}         

Using $\log \langle x \rangle  = \log x$, we compute
\begin{align*}
\left.\frac{d}{ds}\right|_{s=0}\left(\left(1-C\left(\frac{D}{C}\right)\right. \right. & \left. \left. \langle C \rangle^{2s}\right)^{-1}  \int_{\Z_p^\times}  \langle x \rangle^{-s}d\tilde \Psi_\AA^C\right)  \\
&= \left(1 - C\left(\frac{D}{C}\right)\right)^{-1}\int_{\Z_p^\times}\log x \, d\tilde \Psi^C_\AA + (*) \int_{\Z_p^\times} d\tilde \Psi^C_\AA \\
&= \left(1 - C\left(\frac{D}{C}\right)\right)^{-1}\int_{\Z_p^\times}\log x \, d\tilde \Psi^C_\AA 
\end{align*}
The integral $\int_{\Z_p^\times} d\tilde \Psi^C_\AA$ vanishes because by Corollary \ref{Lvanish}, $L_p(f \otimes \chi)(\lambda) = 0$ for all anticyclotomic $\lambda$, in particular for $\lambda = 1$.

If we set $$G_\sigma = (-1)^r \int_{\Z_p^\times} \log_p \, d\tilde \Psi_\AA  \in \bar M_{2r} (\Gamma_0(Np^\infty); \Q_p(\chi)),$$ then using the identity $$\int_{\Z_p^\times} \lambda(\beta) \, d\tilde \Psi^C_\AA = \int_{\Z_p^\times} \lambda (\beta) - C\left(\frac{D}{C}\right)\lambda(C^{-2}\beta) \, d\tilde\Psi_\AA,$$
we obtain 
$$L_p'(f \otimes \chi, \mathbbm{1}) = -H_p(f) \sum_{\sigma \in G(H/K)} L_{f_0}(G_\sigma).$$
Define the operator 
\[\FF = \prod_{\p | p} \left(U_p - p^{r-k-1} \chi(\p) \sigma_{\bar \p}\right)^2.\]
Putting together Corollary \ref{fourier} and Propositions \ref{htcoeff} and \ref{mainid}, we obtain

\begin{proposition}
If $p | m$, $(m,N) = 1$ and $r_\AA(m) = 0$, then
\begin{align*}
\left. c_m^\sigma \right| \FF  = (-1)^{k+1} \left(4|D|\right)^{r-k-1}u^2 a_m(G_\sigma)\bigg| \left(U_p^4 - p^{2r -2}U_p^2\right).
\end{align*}
\end{proposition} 
We define the $p$-adic modular form
$$H_\sigma = F_\AA | \FF +(-1)^k \left(4|D|\right)^{r-k-1}u^2  G_\sigma \bigg| \left(U^4_p - p^{2r-2}U_p^2\right).$$ 
By construction, when $p | m$, $(m,N) = 1$ and $r_\AA(m) = 0$, we have
$$a_m(H_\sigma) = d_m^\sigma | \FF = \chi(\bar\a)^{-1}\sum_{v | p} \langle x, T_m \bar x^\a\rangle_v | \FF.$$

\begin{proposition}\label{localvanishing}
Define the operator
$$\FF' = (U_p - \sigma_\p)(U_p\sigma_\p - p^{2r-2})(U_p - \sigma_{\bar\p})(U_p\sigma_{\bar\p} - p^{2r-2}).$$
Then $L_{f_0}(H_\sigma | \FF') = 0$.
\end{proposition}

\begin{proof}
The proof should be exactly as in \cite[II.5.10]{Nek}, however the proof given there is not correct.  In the next section we explain how to modify Nekov\'a\v{r}'s argument to prove the desired vanishing.  For our purposes in this section, the important point is that this modified proof goes through if we replace the representation $V_{f,A,0} = V_f$ (i.e. the $\ell = 0$ case which Nekov\'a\v{r} considers) with our representation $V_{f,A,\ell} = V_f \otimes W$, where $W$ corresponds to a trivial local system.  Indeed, the proof works ``on the curve" and essentially ignores the local system.  The only inputs specific to the local system are two representation-theoretic conditions: it suffices to know that the representation $V_{f,A,\ell}$ is ordinary and crystalline.  These follow from Theorems \ref{ord} and \ref{AJ}, respectively.    
\end{proof}

It follows that
$$L_{f_0}\left(F_\AA  | \FF \FF'\right) = (-1)^{k+1}\left(4|D|\right)^{r-k-1}u^2 L_{f_0}\left(G_\sigma \bigg| \left(U^4_p - p^{2r-2}U^2_p\right)\FF'\right).$$
Since $L_{f_0} \circ U_p  = \alpha_p(f)L_{f_0}$, we can remove $\FF'$ from the equation above; we may divide out the extra factors that arise as they are non-zero by the Weil conjectures.  Summing this formula over $\sigma \in \Gal(H/K)$, we obtain
\begin{align*}
L_{f_0}(f) &\prod_{\p | p} \left(1 - \frac{\chi(\p)p^{r-k-1}}{\alpha_p(f)}\right)^2\sum_{\sigma \in \Gal(H/K)}  \langle z_f, z_{f,\bar\chi}^\AA\rangle \\
  &= (-1)^k \left(4|D|\right)^{r-k-1}u^2 H_p(f)^{-1}\left(1 - \frac{p^{2r-2}}{\alpha_p(f)^2}\right) L_p'(f \otimes \chi, \mathbbm{1}).
\end{align*}
Note that the operators $\sigma_\p$ and $\sigma_{\bar \p}$ (in the definition of $\FF$) permute the various $\langle z_f, z_{f,\bar\chi}^\AA\rangle$ as $\AA$ ranges through the class group.  So after summing over $\Gal(H/K)$, these operators have no effect and therefore do not show up in the Euler product in the left hand side.\footnote{This is unlike what happens in \cite{Nek}.  The difference stems from the fact that we inserted the Hecke character into the definition of the measures defining the $p$-adic $L$-function.}    
By Hida's computation \cite[I.2.4.2]{Nek}:
$$\left(1 - \frac{p^{2r-2}}{\alpha_p(f)^2}\right)= H_p(f)L_{f_0}(f),$$
so we obtain
\begin{equation*}L'_p(f \otimes \chi, \mathbbm{1}) =(-1)^k\prod_{\p | p}\left(1 - \frac{\chi(\p)p^{r-k-1}}{\alpha_p(f)} \right)^2 \frac{\sum_{\AA \in \Pic(\O_K)} \langle z_f, z_{f,\bar\chi}^\AA \rangle}{\left(4|D|\right)^{r-k-1}u^2}.
\end{equation*}
By equation (\ref{ortho}), this equals 
\[ (-1)^k\prod_{\p | p}\left(1 - \frac{\chi(\p)p^{r-k-1}}{\alpha_p(f)} \right)^2 \frac{h \langle z_{f,\chi} z_{f,\bar\chi}^\AA \rangle}{\left(4|D|\right)^{r-k-1}u^2}\]
proves Theorem \ref{main}.  

\begin{proof}[Proof of Theorem $\ref{PRproof}$]
We now assume $\chi = \psi^\ell$ as in Section \ref{apps}. Recall that the cohomology classes $z_f$ and $\bar z_f$ live in $H^1_f(H, V_{f,A,\ell})$.  Recall also $V_{f,A,\ell}$ is the 4-dimensional $p$-adic realization of the motive $M(f)_H \otimes M(\chi_H)$ over $H$ with coefficients in $\Q(f)$.  Using Remark \ref{descend}, we have a motive $M(f)_K \otimes M(\chi)$ over $K$ with coefficients in $\Q(f,\chi)$ descending $M(f)_H \otimes M(\chi_H) \otimes \Q(\chi)$.  The $p$-adic realization of this motive over $K$ is what we called $V_{f,\chi}$.     

Thus we may think of the classes $z_f$ and $\bar z_f$ in $H^1_f(H, V_{f,A,\ell}) \cong H^1(H, V_{f,\chi})$.  Define
\[z^K_f = \cor_{H/K}(z_f) \hspace{5mm} \mbox{and} \hspace{5mm} \bar z^K_f = \cor_{H/K}(\bar z_f)\]
in $H^1_f(K, V_{f,\chi})$.  

\begin{lemma}\label{cores}
\[\res_{H/K}(z_f^K) = hz_{f,\chi}\hspace{5mm} \mbox{and} \hspace{5mm} \res_{H/K}(\bar z_f^K) = h z_{f, \bar\chi}.\]
\end{lemma}

\begin{proof}
Note that there is a natural action of $\Gal(H/K)$ on $H^1(H, V_{f,\chi})$, since $V_{f,\chi}$ is a $G_K$-representation.  Since $\res \circ \cor = \Nm$, it suffices to show that for each $\sigma \in \Gal(H/K)$, $z_f^\sigma = z_{f,\chi}^\AA$ and $\bar z^\sigma_f = z_{f, \bar \chi}^\AA$, where $\AA$ corresponds to $\sigma$ under the Artin map.  Recall that 
\[z_{f,\chi}^\AA = \chi(\a)^{-1} \Phi_f\left(\e_B\e Y^\a\right) \hspace{5mm} \mbox{and} \hspace{5mm} z_{f,\bar \chi}^\AA = \chi(\bar \a)^{-1} \Phi_f\left(\e_B\bar\e Y^\a\right),\] for any ideal $\a$ in the class of $\AA$.

To prove $z_f^\sigma = z_{f,\chi}^\AA$, we first describe explicitly the action of $\Gal(\bar K/K)$ on the subspace $\e V_{f,A,\ell} \subset V_{f,A,\ell}$, after identifying the spaces $V_{f,A,\ell}$ and $V_{f,\chi}$.  For each $\sigma \in \Gal(\bar K/K)$, we have maps
\[\e_\ell H^\ell(\bar A^\ell, \Q_p) \stackrel{\sigma^*}{\longrightarrow} \e_\ell^\sigma H^\ell(\overline{A^\sigma}^\ell, \Q_p) \xrightarrow{\chi(\a)^{-1}\phi_\a^{\ell *}}\e_\ell H^\ell(\bar A^\ell, \Q_p),\]
which induces an action of $G_K$ on $\e V_{f,A,\ell} = V_f \otimes \e H^\ell(\bar A^\ell, \Q_p(k))$.  By definition of $M(\chi)$, this agrees with the action of $G_K$ on $V_{f,\chi}$.  Now the argument in the proof of Lemma \ref{invariance} shows that $z_f^\sigma = z_{f,\chi}^\AA$.  A similar argument works for $\bar z_f^\sigma$.    
\end{proof}

By Lemma \ref{cores}, $\res_{H/K}(z^K_{f,\chi} ) = hz_{f,\chi}$ and $\res_{H/K}(\bar z^K_f) = h z_{f, \bar\chi}$. It follows that 
\begin{equation}\label{hteq}
\left\langle z^K_f,\bar z^K_f \right\rangle_K = h \left\langle z_{f,\chi} , z_{f,\bar\chi}\right\rangle_H.
\end{equation} 

Now assume that $L'_p(f \otimes \chi, \ell_K, \mathbbm{1}) \neq 0$.  By Theorem \ref{main} and (\ref{hteq}), the cohomology classes $z_f^K$ and $\bar z_f^K$ are non-zero, giving two independent elements of $H^1_f(K, V_{f,\chi}).$ This proves one inequality in Perrin-Riou's conjecture (\ref{PRconj}).       
%
The other inequality follows from forthcoming work of Elias \cite{yara} constructing an Euler system of generalized Heegner classes and extending the methods of Kolyvagin and Nekov\' a\v r in \cite{NekEuler} to our setting (see also \cite[Theorem B]{hsieh}).  
\end{proof}


\section{Local $p$-adic heights at primes above $p$}\label{nekfix}

The purpose of this last section is to fix the proof of \cite[II.5.10]{Nek} on which both Nekov\'a\v{r}'s Theorem A and our main theorem rely.  In the first two subsections we gather some facts about relative Lubin-Tate groups and ring class field towers, and in \ref{fixes} we explain how to modify the proof in \cite{Nek}.  We have isolated and fixed two arguments of \cite[II.5]{Nek}, instead of rewriting the entire argument of that section.

\subsection{Relative Lubin-Tate groups}
The reference for this material is \cite[\S 1]{dS}.  

Let $F/\Q_p$ be a finite extension and let $L$ be the unramified extension of $K$ of degree $\delta \geq 1$.  Write $\m_F$ and $\m_L$ for the maximal ideals in $\O_F$ and $\O_L$ and write $q$ for the cardinality of $\O_F/\m_F$.  We let $\phi: L \to L$ be the Frobenius automorphism lifting $x \to x^q$ and normalize the valuation on $F$ so that a uniformizer has valuation 1.  Let $\xi \in F$ be an element of valuation $\delta$ and let $f \in \O_{L}[[X]]$ be such that 
\[f(X) = \varpi X + O(X^2) \hspace{4mm} \mbox{and} \hspace{4mm} f(X) \equiv X^q \,  \mod \m_L,\]
where $\varpi \in \O_L$ satisfies $\Nm_{L/F}(\varpi) = \xi$.  Note that $\varpi$ exists and is a uniformizer, since $\Nm_{L/F}(L^\times)$ is the set of elements in $F^\times$ with valuation in $\delta\Z$.

\begin{theorem}
There is a unique one dimensional formal group law $F_f \in \O_L[[X,Y]]$ for which $f$ is a lift of Frobenius, i.e.\ for which $f \in \Hom(F_f, F_f^\phi)$.  $F_f$ comes equipped with an isomorphism $\O_F \cong \End(F_f)$ denoted $a \mapsto [a]_f$, and the isomorphism class of $F_f/\O_L$ depends only on $\xi$ and not on the choice of $f$.        
\end{theorem}          

Now let $M$ be the valuation ideal of $\C_p$ and let $M_f$ the $M$-valued points of $F_f$.  For each $n \geq 0$, the $\m_F^n$-torsion points of $F_f$ are by definition
\[W_f^n = \{ \omega \in M_f : \,  [a]_f(\omega) = 0 \hspace{3mm}  \mbox{for all} \hspace{3mm} a \in \m_F^n\} \]

\begin{proposition}\label{relLT}
For each $n \geq 1$, set $L^n_\xi = L(W_f^n)$.  Then
\begin{enumerate} 
\item $L_\xi^n$ is a totally ramified extension of $L$ of degree $(q-1)q^{n-1}$ and is abelian over $F$.  
\item There is a canonical isomorphism $(\O_F/\m_F^n)^\times \cong \Gal(L^n_\xi/L)$ given by $u \mapsto \sigma_u$, where $\sigma_u(\omega) = [u^{-1}]_f(\omega)$ for $\omega \in W_f^n$.  
\item Both the field $L^n_\xi$ and the isomorphism above are independent of the choice of $f$.  
\item The map $u \mapsto \sigma_u$ is compatible with the local Artin map $r_F: F^\times \to \Gal(F^\ab/F)$. 
\item The field $L^n_\xi$ corresponds to the subgroup $\xi^\Z\cdot \left(1 + \m_F^n\right) \subset F^\times$ via local class field theory.   
\end{enumerate}
\end{proposition}

Writing $L_\xi = \bigcup_n L_\xi^n$, we see that $\Gal(L_\xi/L) \cong \O_F^\times$ and the group of universal norms in $F^\times$ coming from $L_\xi$ is $\xi^\Z$.  Moreover, we have an isomorphism $\Gal(L_\xi/ L) \to \O_F^\times$ who's inverse is $r_F|_{\O_F^\times}$ composed with the restriction $\Gal(F^\ab/F) \to \Gal(L_\xi/F)$.

\subsection{Relative Lubin-Tate groups and ring class field towers}

Now let $v$ be a place of $H$ above $p$ and above the prime $\p$ of $K$.  For each $j \geq 1$, write $H_{j,w}$ for the completion of the ring class field $H_{p^j}$ of conductor $p^j$ at the unique place $w = w(j)$ above $v$.  In particular, $H_{0,v} = H_v$.  If $\delta$ is the order of $\p$ in $\Pic(\O_K)$, then $H_v$ is the unramified extension of $K_\p \cong \Q_p$ of degree $\delta$.  Since $p$ splits in $K$, $H_{j,w}/H_v$ is totally ramified of degree $(p-1)p^{j-1}/u$, where recall $u = \#\O_K^\times /2$.  Moreover, $\Gal(H_{j,w}/H_v)$ is cyclic and $H_{j,w}$ is abelian over $\Q_p$.  We call $H_\infty = \bigcup_j H_{j,w}$ the local ring class field tower; it contains the anticyclotomic $\Z_p$-extension of $K_\p$.  To ease notation and to recall the notation of the previous section, we write $L = H_v$.               

\begin{proposition}\label{unif}
Write $\p^\delta = (\pi)$ for some $\pi \in \O_K$.  Then $H_\infty$ is contained in the field $L_\xi$ attached to the Lubin-Tate group relative to the extension $L/\Q_p$ with parameter $\xi = \pi/\bar\pi$ in  $K_\p \cong \Q_p$.  If $\O_K^\times = \{\pm 1\}$, then $H_\infty = L_\xi$.  
\end{proposition}

\begin{remark}
Note that there are other natural Lubin-Tate groups relative to $L/\Q_p$ coming from the class field theory of $K$, namely the formal groups of elliptic curves with complex multiplication by $\O_K$.  These formal groups will have different parameters however, as can be seen from the discussion in \cite[II.1.10]{dS}.    
\end{remark}

\begin{proof}
By $(5)$ of Proposition \ref{relLT}, it is enough to prove that $H_\infty$ is the subfield of $\Q_p^\ab$ corresponding to the subgroup $(\pi/\bar\pi)^\Z\cdot \mu_K^2$ under local class field theory.  First we show that $(\pi/\bar\pi)$ is norm from every $H_{j,w}$.  Using the compatibility between local and global reciprocity maps, this will follow if the idele (with non-trivial entry in the $\p$ slot) 
\[(\ldots 1\, , 1\, , \pi/\bar\pi\, , 1,\,  1\,, \ldots) \in \A^\times_K\]
is in the kernel of the reciprocity map 
\[r_j: \A^\times_K/K^\times \to \Gal(K^\ab/K) \to \Gal(H_{p^j}/K),\]
for each $j$.  Since the kernel of $r_j$ is $K^\times\A_{K,\infty}^\times\hat \O_{p^j}^\times$, it is enough to show that
\[(\ldots 1/\pi\, , 1/\pi\, , 1/\bar\pi\,, 1/\pi\,, 1/\pi\,, \ldots) \in \hat\O_{p^j}^\times.\] 
This is clear at all primes away from $p$ since $\pi$ is a unit at those places.  At $p$, it amounts to showing that 
$(1/\bar \pi, 1/\pi) \in K_\p \times K_{\bar \p}$ lands in the diagonal copy of $\Z_p$ under the identification 
$K_\p \times K_{\bar \p} \cong \Q_p \times \Q_p$, and this is also clear.    

Since $L/\Q_p$ is unramified of degree $\delta$ and $\xi =\pi/\bar\pi$ has valuation $\delta$, it remains to prove that the only units in $\Q_p$ which are universal norms for the tower $H_\infty/\Q_p$ are those in $\mu_K^2$.  But by the same argument as above, the only way $\alpha \in \Z_p^\times$ can be a norm from every $H_{j,w}$ is if $\alpha \zeta = \bar \zeta$ for some global unit $\zeta \in K$.  But then $\zeta$ is a root of unity and $\alpha = \zeta^{-1}\bar\zeta = \zeta^{-2}$, so $\alpha$ is in $\mu_K^2$.  Conversely, it's clear that each $\zeta \in \mu_K^2$ is a universal norm.     
\end{proof}

\begin{remark}
Since we are assuming $K$ has odd discriminant, the equality $H_\infty = L_\xi$ holds unless $K = \Q(\mu_3)$.  For ease of exposition we will assume $K \neq \Q(\mu_3)$ for the rest of this section; the modifications needed for the case $K = \Q(\mu_3)$ are easy enough.       
\end{remark}

We will need one more technical fact about the relative Lubin-Tate group $F_f$ cutting out $H_\infty$.  Let $\chi_\xi: \Gal(\bar L/L) \to \Z_p^\times$, be the character giving the Galois action on the torsion points of $F_f$.  We let $\Q_p(\chi_\xi)$ denote the 1-dimensional $\Q_p$-vector space endowed with the action of $\Gal(\bar L/L)$ determined by $\chi_\xi$, and we denote by $D_\cris(\Q_p(\chi_\xi))$ the usual filtered $\phi$-module contravariantly attached to the $\Gal(\bar L/L)$-representation $\Q_p(\chi_\xi)$ by Fontaine.   

\begin{proposition}\label{LTcrys}
The representation $\Q_p(\chi_\xi)$ is crystalline and the frobenius map on the $1$-dimensional $L$-vector space $D_\cris(\Q_p(\chi_\xi))$ is given by multiplication by $\xi$.  
\end{proposition}   

\begin{proof}
This is presumably well known, but with a lack of reference we will verify this fact using \cite[Prop.\ B.4]{conrad}. There it is shown that $\Q_p(\chi_\xi)$ is crystalline if and only if there exists a homomorphism of tori $\chi' : L^\times \to \Q_p^\times$ which agrees with the restriction of $\chi_\xi \circ r_L$ to $\O_L^\times$.  In that case, frobenius on $D_\cris(\Q_p(\chi_\xi))$ is given by multiplication by $\chi_\xi(r_L(\varpi)) \cdot \chi'(\varpi)^{-1}$, where $\varpi$ is any uniformizer of $L$.\footnote{Note that we are using the contravariant $D_\cris$, whereas \cite{conrad} uses the covariant version.}  Combining (2) and (4) of Proposition \ref{relLT} with the commutativity of the following diagram
       
\[\begin{CD}
L^\times @>r_L>>  \Gal(L^\ab/L) \\
@V\Nm VV          @VV\res V  \\
\Q_p^\times  @>r_{\Q_p}>> \Gal(\Q_p^\ab/\Q_p),
\end{CD}\]
we see that $\chi' = \Nm^{-1}$ gives such a homomorphism, so the crystallinity follows.  Note that by construction $\chi_\xi: \Gal(L^\ab/L) \to \Z_p^\times$ factors through a character $\tilde\chi_\xi: \Gal(\Q_p^\ab/L) \to \Z_p^\times$.  So if we choose $\varpi$ to be such that $\Nm_{L/\Q_p}(\varpi) = \xi$, then 
\begin{align*}
\chi_\xi(r_L(\varpi)) &= \tilde \chi_\xi(r_{\Q_p}(\Nm(\varpi)))\\
&= \tilde\chi_\xi(r_{\Q_p}(\xi)) = 1.      
\end{align*} 
Thus, the frobenius is given by multiplication by $\chi'(\varpi)^{-1} = \Nm_{L/\Q_p}(\varpi) = \xi$.        
\end{proof}

\subsection{Local heights at $p$ in ring class field towers}\label{fixes}
The proofs of both \cite[II.5.6]{Nek} and \cite[II.5.10]{Nek} mistakenly assert that $H_{j,w}$ contains the $j$-th layer of the cyclotomic $\Z_p$-extension of $\Q_p$ (as opposed to the anticyclotomic $\Z_p$-extension).  This issue first arises in the proofs of \cite[II.5.9]{Nek} and \cite[II.5.10]{Nek}.  We explain now how to adjust the proof of \cite[II.5.10]{Nek}; similar adjustments may be used to fix the proof of \cite[II.5.9]{Nek}.  Our approach is in the spirit of Nekov\'{a}\v{r}'s original argument, but uses extra results from $p$-adic Hodge theory to carry the argument through.  

Recall the setting of \cite[II.5.10]{Nek}: $x$ is the Tate vector corresponding to our (generalized) Heegner cycle $\e_B \e Y$, and $V = H^1_\et(\bar X_0(N),j_{0*}\AA)(1)$.  We have the Tate cycle
\[x_f = \sum_{m \in S} c_{f, m} T_m x \in Z(Y_0(N),H) \otimes_{\Q_p} L,\]
a certain linear combination (with coefficients $c_{f,m}$ living in a large enough field $L$) of $T_m x$ such that\[\Phi_T(x_f) = z_f \in H^1_f(H,V) \otimes_{\Q_p} L.\]  Moreover, each $m \in S$ satisfies $(m,pN) = 1$ and $r(m) = 0$, where $r(m)$ is the number of ideals in $K$ of norm $m$.  To fix the proof of \cite[II.5.10]{Nek}, we prove the following vanishing result for local heights at primes $v$ of $H$ above $p$. 
 
\begin{theorem}\label{vanishing}
For each $j \geq 1$, let $h^\sigma_j \in Z_f(Y_0(N), H_{j,w})$ be a Tate vector supported on a point $y_j \in Y_0(N)$ corresponding to an elliptic curve $E_j$ such that $\End(E_j)$ is the order in $\O_K$ of index $p^j$.    Then 
\[ \lim_{j \to \infty} \langle x_f, N_{H_{j,w}/H_v}(h^\sigma_j) \rangle_v = 0.\]
 \end{theorem}
\begin{proof}
Recall that $E_j$ is a quotient of an elliptic curve $E$ with CM by $\O_K$ by a (cyclic) subgroup of order $p^j$ which does not contain either the canonical subgroup $E[\p]$ or its dual $E[\bar \p]$.  By the compatibility of local heights with norms \cite[II.1.9.1]{Nek}, we have
\begin{equation}\label{pairing} \left\langle x_f, N_{H_{j,w}/H_v}(h^\sigma_j)\right\rangle_{v, \ell_v} = \left\langle x_f, h^\sigma_j \right\rangle_{w, \ell_w},\end{equation}
where $\ell_w = \ell_v \circ N_{H_{j,w}/H_v}$.  Recall that we are assuming now that $\ell_K = \log_p \circ \lambda$, where $\l: \Gal(K_\infty/K) \to 1 + p\Z_p$ is the cyclotomic character.  Thus the local component $\ell_v : H_v^\times \to \Q_p$ of $\ell_K$ is $\ell_v = \log_p \circ N_{H_v/\Q_p}$, and 
\[\ell_w = \log_p \circ N_{H_{j,w}/\Q_p}.\]  

We have seen that the ring class field tower $H_\infty$ is cut out by a relative Lubin-Tate group.  In fact, it follows from the results in the previous sections that $H_{j,w} = L^j_\xi$, where $L = H_v$ and $\xi = \pi/\bar\pi$ as before.  Let $E$ be the mixed extension used to compute the height pairing of $x_f$ and $h^\sigma_j$ (as in \cite[II.1.7]{Nek}), and let $E_w$ be its restriction to the decomposition group at $w$.  Assume that 
\[\mbox{$E_w$ is a crystalline representation of $\Gal(\bar H_{j,w}/H_{j,w})$}.\]  
Then by definition of the local height, we have 
\begin{align*}
\left\langle x_f, h^\sigma_j \right\rangle_{w, \ell_w} &= \ell_w(r_w([E_w]))\\
&=\log_p\left(N_{H_{j,w}/\Q_p}(r_w([E_w]))\right).
\end{align*}
where $r_w([E_w])$ is an element of $\widehat{\O_{H_{j,w}}^\times} \otimes_{\Z_p} \Q_p$.  In fact, the ordinarity of $f$ allows Nekov\'{a}\v{r} to ``bound denominators"; i.e. he shows  
\[p^{-d_j}\left\langle x_f, h^\sigma_j \right\rangle_{w, \ell_w}  \in \log_p\left(N_{H_{j,w}/\Q_p}\left(\widehat{\O_{H_{j,w}}^\times}\right)\right) .\]
for some integer $d_j$.  Indeed, this follows from our assumption that $E_w$ is crystalline and the proofs in \cite[II.1.10, II.5.10]{Nek}; note that $H^1_f(H_{j,w},\Z_p(1)) = \widehat{\O_{H_{j,w}}^\times}$.  Moreover, the $d_j$ are uniformly bounded as $j$ varies.  Nekov\'{a}\v{r}'s proof of this last fact does not quite work, but we fix this issue in Proposition \ref{hodgetate} below.  Let us write $d = \sup_j d_j$.  By Proposition \ref{relLT},  we have
\[p^{-d}\left\langle  x_f, h^\sigma_j \right\rangle_{w, \ell_w}  \in \log_p(1 + p^j\Z_p) \subset p^j\Z_p.\]
The theorem would then follow upon taking the limit as $j \to \infty$.

It therefore remains to show that $E_w$ is crystalline. First we need a lemma. 

\begin{lemma}\label{disjoint}
Let $m \in S$ and $j$ be as above.  Then the supports of $T_m x$ and $b^\sigma_{p^j}$ are disjoint on the generic and special fibers of the integral model $\X$ of $X_0(N)$.   
\end{lemma}  
\begin{proof}
Let $z \in Y_0(N)(\bar\Q_p)$ be in the support of $T_m x$ and let $y$ be the Heegner point supporting the Tate cycle $x$.  Thinking of these points as elliptic curves via the moduli interpretation, there is an isogeny $\phi: y \to z$ of degree prime to $p$ since $(p,m) = 1$.  Recall $p$ splits in $K$, so that $y$ has ordinary reduction $y_s$ at $v$.  Since $\End(y) \cong \O_K \cong \End(y_s)$, $y$ is a Serre-Tate canonical lift of $y_s$.  As $\phi$ induces an isomorphism of $p$-divisible groups, $z$ is also a canonical lift of its reduction.  On the other hand, the curve $E_j$ supporting $h_j^\sigma$ has CM by a non-maximal order of $p$-power index in $\O_K$ and is therefore not a canonical lift of its reduction.  Indeed, the reduction of $E_j$ is an elliptic curve with CM by the full ring $\O_K$ as it obtained by successive quotients of $y_s$ by either the kernel of Frobenius or Verschiebung.  This shows that $T_mx$ and $b^\sigma_{p^j}$ have disjoint support in the generic fiber.

By \cite[III.4.3]{GZ}, the divisors $T_{mn} y$ and $y^\tau$ are disjoint in the generic fiber, for any $\tau \in \Gal(H/K)$.  Since all points in the support of these divisors are canonical lifts, the divisors must not intersect in the special fiber either.  But we saw above that the special fiber of $E_j$ is a Galois conjugate of the reduction of $y$, so $E_j$ and $T_my$ are disjoint on the special fiber as well.                 
\end{proof}

Next we note that $T_m x$ is a sum $\sum d_i$, where each $d_i$ is supported on a single closed point $S$ of $Y_0(N)/H_{j,w}$.  Using norm compatibility once more and base changing to an extension $\F/H_{j,w}$ which splits $S$, we may assume that $S \in Y_0(N)(\F)$.  

It then suffices to show that the mixed extension $E_w'$ corresponding to $d_i$ and $h^\sigma_j$ is crystalline.  Recall from \cite[II.2.8]{Nek} that this mixed extension is a subquotient of 
\[H^1(\bar X_0(N) - \bar S\, \rel  \bar T, j_{0*}\AA)(1),\] 
where $T = y_j$ is the point supporting $h_j^\sigma$.  So it is enough to show that this cohomology group is itself crystalline.  Finally, this follows from combining the previous lemma with the following result.
\end{proof} 


\begin{theorem}\label{crysmixed}
Suppose $\F$ is a finite extension of $\Q_p$ and let $S,T \in Y_0(N)(\F)$ be points with non-cuspidal reduction and which do not intersect in the special fiber.  Then $H^1(\bar X_0(N)  - \bar S \rel \bar T, j_{0*}\AA)(1)$ is a crystalline representation of $G_\F$.      
\end{theorem}

\begin{remark}
Suppose $F$ is a $p$-adic field and $X/\Spec \O_F$ is a smooth projective variety of relative dimension $2k-1$.  If $Y, Z \subset X$ are two (smooth) subvarieties of codimension $k$ which do not intersect on the special fiber, then one expects that $H^{2k-1}(\bar X_F - \bar Y_F \rel \bar Z_F, \Q_p(k))$ is a crystalline representation of $G_F$.  The theorem above proves this for cycles sitting in fibers of a map $X \to C$ to a curve.  The general case should follow from the machinery developed in the recent preprint \cite{niziol}.   
\end{remark}

\begin{proof}
Write $\V = H^1(\bar X_0(N) - \bar S \rel T, j_{0*}\AA)(1)$.  The sketch of the proof is as follows.  Faltings' comparison isomorphism \cite{Falt} identifies $D_\cris(\V)$ with the crystalline analogue of $\V$, which we will refer to (in this sketch)  as $H^1_\cris(X - S \rel T, j_{0*}\AA)$.  The dimension of $\V$ is determined by the standard exact sequences
\begin{equation}\label{gysin}
0 \to H^0(\bar T,  j_{0*}\AA)(1) \to \V \to H^1(\bar X -\bar S,  j_{0*}\AA)(1) \to 0\end{equation}
\[0 \to H^1(\bar X,  j_{0*}\AA)(1) \to H^1(\bar X - \bar S,  j_{0*}\AA)(1) \to H^0(\bar S,  j_{0*}\AA) \to 0\]  
Similar exact sequences should hold in the crystalline theory (i.e.\ with $H^1$ replaced by $H^1_\cris$ everywhere) since $S$ and $T$ reduce to distinct points on the special fiber.  Using the known crystallinity of $H^1(\bar X, j_{0*}\AA)(1)$, $H^0(\bar T,  j_{0*}\AA)(1)$, and  $H^0(\bar S,  j_{0*}\AA)$ (the latter two because the fibers of $X \to X(N)$ above $S$ and $T$ have good reduction), we conclude that 
\[\dim_{\Q_p} \V = \dim_{F_0} H^1_\cris(X - S \rel T, j_{0*}\AA),\] i.e. that $\V$ is crystalline.  To turn this sketch into a proof, we need to say explicitly what $H^1_\cris(X - S \rel T, j_{0*}\AA)$ is.  Note that the usual crystalline cohomology is not a good candidate because it is not usually finite dimensional unless the variety is smooth and projective.  

Let us describe in more detail the comparison isomorphism which we invoked above.  The main result of \cite{Falt} concerns the cohomology of a smooth projective variety with trivial coefficients.  In our setting, however, we deal with cohomology of an affine curve with partial support along the boundary and with non-trivial coefficients.  The proof of the comparison isomorphism in this more complicated situation is sketched briefly in \cite{Falt} as well, but we follow the exposition \cite{ols}, where the modifications we need are explained explicitly  and in detail.  

Let $R$ be the ring of integers of $\F$ and set $V = \Spec(R)$.  Let $X/V$ be a smooth projective curve and let $S,T \in X(V)$ be two rational sections which we think of as divisors on $X$.  We assume that $S$ and $T$ do not intersect, even on the closed fiber.  Set $D = S \cup T$ and $X^o = X - D$.  The divisor $D$ defines a log structure $M_X$ on $X$ and we let $(Y,M_Y)$ be the closed fiber of $(X,M_X)$.  We use the log-convergent topos $((Y,M_Y)/V)_\conv$ to define the `crystalline' analogue of $\V$.  There is an isocrystal $J_S$ on $((Y,M_Y)/V)_\conv$ which is \'etale locally defined by the ideal sheaf of $S$; see \cite[\S13]{ols} for its precise definition and for more regarding the convergent topos.

\begin{theorem}[Faltings, Olsson]
Let $L$ be a crystalline sheaf on $X^o_\F$ associated to a filtered isocrystal $(F, \varphi_F, Fil_\FF)$.  Then there is an isomorphism
\begin{equation}
B_\cris(\bar V) \otimes_\F H^1(((Y,M_Y)/V)_\conv, F \otimes J_S) \to B_\cris(\bar V) \otimes_{\Q_p} H^1(\bar X - \bar S \rel \bar T, L).
\end{equation}     
\end{theorem}

As $L = j_{0*}\AA$ is crystalline \cite[6.3]{Falt}, we may apply this theorem in our situation.  Taking Galois invariants, we conclude that $D_\cris(\V) = H^1(((Y, M_Y)/V)_\conv, F \otimes J_S)$.  To complete the proof of Theorem \ref{crysmixed}, it would be enough to know that the convergent cohomology group $D_\cris(\V)$ sits in exact sequences analogous to the standard Gysin sequences (\ref{gysin}).  These sequences hold in any cohomology theory satisfying the Bloch-Ogus axioms, but unfortunately convergent cohomology is not known to satisfy these axioms.  On the other hand, rigid cohomology does satisfy the Bloch-Ogus axioms \cite{petrequin}.  So we apply Shiho's log convergent-rigid comparison isomorphism \cite[2.4.4]{shiho} to identify $D_\cris(\V)$ with $H^1_\rig(Y - S_s \rel T_s, j^\dagger \E)$, for a certain overconvergent isocrystal $j^\dagger\E$ which is the analogue of $j_{0*}\AA$ on the special fiber.  Here $S_s$ and $T_s$ are the points on the special fiber.  We have similar identifications with rigid cohomology for each term appearing in the sequences (\ref{gysin}), and the corresponding short exact sequences of rigid cohomology groups are exact.  The crystallinity of $\V$ now follows from dimension counting.  
\end{proof}

\begin{remark}
Theorem \ref{vanishing} has two components: first one must bound denominators and then one shows that the heights go to 0 $p$-adically.  In the argument above, the ordinarity of $f$ was the crucial input needed to bound denominators.  We briefly explain the modifications need to fix the proof of \cite[II.5.9]{Nek}, where one pairs Heegner cycles of $p$-power conductor with cycles in the kernel of the local Abel-Jacobi map (the higher weight analogue of principal divisors).  The fact that these cycles are Abel-Jacobi trivial allows us to make a ``bounded denominators" argument even without an ordinarity assumption; see \cite[II.1.9]{Nek}.  To kill the $p$-adic height, we further note that the particular AJ-trivial cycles in the proof of II.5.9 are again linear combinations of various $T_n x$, with $r(n) = 0$.  This lets us invoke Lemme \ref{disjoint} and Theorem \ref{crysmixed}, as before.     
\end{remark}

As we alluded to in the proof of Theorem \ref{vanishing}, the proof of \cite[II.5.11]{Nek} again assumes that $H_\infty$ contains the cyclotomic $\Z_p$-extension of $\Q_p$.  To fix the proof there, it is enough to prove the following proposition.  

\begin{proposition}\label{hodgetate}
Let $V$ be the Galois representation $H^1_\et(\bar X_0(N),j_{0*}\AA)(1)$ attached to weight $2r$ cusp forms.  If we set $H_{\infty} = \bigcup_j H_{j,w}$, then \[H^0(H_\infty, V) = 0.\]  
\end{proposition}

\begin{proof}
We follow Nekov\'{a}\v{r}'s approach, but instead of using the cyclotomic character we use the character $\chi_\xi$ coming from the relative Lubin-Tate group attached to $H_\infty$, defined above.  By Proposition \ref{LTcrys}, the $G_{\Q_p}$-representation $\Q_p(\chi_\xi)$ is crystalline and the frobenius on $D_\cris(\Q_p(\chi_\xi))$ is given by multiplication by $\xi$, where $\xi$ is defined in Proposition \ref{unif}. 

Since $V$ is Hodge-Tate, there is an inclusion of $\Gal(H_\infty/H_v)$-representations  
\[H^0(H_\infty, V) \subset \oplus_{j \in \Z} H^0(H_v, V(\chi_\xi^j))(\chi_\xi^{-j}).\]
Indeed, $H^0(H_\infty, V)$ has an action by $\Gal(H_\infty /H)$ which we can break up into isotypic parts indexed by characters $\chi_\xi^s$, with $s \in \Z_p$.  But of these characters, the only ones which are Hodge-Tate are those with $s \in \Z$, so we obtain the inclusion above.    

So it suffices to show that for each $j$, $H^0(H_v, V(\chi_\xi^j))(\chi_\xi^{-j}) = 0$.  Tensoring the inclusion $\Q_p \to B_\cris^{f = 1}$ by $V(\chi_\xi^j)$, taking invariants, and then twisting the resulting filtered frobenius modules by $\chi_\xi^{-j}$, we obtain 
\[H^0(H_v, V(\chi_\xi^j))(\chi_{\xi}^{-j}) \subset D_\cris(V)^{f = \xi^{-j}}.\]
As an element of $\C$, $\xi$ has absolute value 1.  Since $V$ appears in the odd degree cohomology of the Kuga-Sato variety, \cite{KM} implies that $D_\cris(V)^{f = \xi^{-j}}$ vanishes and the proposition follows.  
\end{proof}

Finally, for completeness, we explain how Proposition \ref{hodgetate} is used in the proof of Proposition \ref{localvanishing}.  Let $X$ be the (generalized) Kuga-Sato variety over $H_v$ and let $T$ be the image of the map $$H^{2r + 2k  -1}(\bar X, \Z_p(r+k)) \to V = H^{2r +2k-1}(\bar X, \Q_p(r+k)).$$
Proposition \ref{hodgetate} is used to infer the following fact, whose proof was left to the reader in \cite{Nek}. 

\begin{proposition}
The numbers $\#H^1(H_{j,w}, T)_\tors$ are bounded as $j \to \infty$.   
\end{proposition}

\begin{proof}
From the short exact sequence
$$0 \to T \to V \to V/T \to 0,$$
we have $$(V/T)^{G_j} \to H^1(G_j, T) \to H^1(G_j, V) \to 0,$$ where $G_j = \Gal(\bar H_{j,w}/H_{j,w}).$  As $H^1(G_j, V)$ is torsion-free, we see that $(V/T)^{G_j}$ maps surjectively onto $H^1(G_j,T)_\tors$.  An element of order $p^a$ in $(V/T)^{G_j}$ is of the form $p^{-a}t$ for some $t \in T$ not divisible by $p$ in $T$.  We then have $\sigma t - t \in p^a T$ for all $\sigma \in G_j$.  As $V/T \cong (\Q_p/\Z_p)^n$ for some integer $n$, it suffices to show that $a$ is bounded as we vary over all elements of $(V/T)^{G_j}$ and all $j$.  

Suppose these $a$ are not bounded.  Then we can find a sequence $t_i \in T$ such that $t_i \notin pT$ and such that $\sigma t_i - t_i \in p^{a(i)}T$ for all $\sigma \in G_\infty := \Gal(\bar H/H_\infty)$.  Here, $a(i)$ is a non-decreasing sequence going to infinity with $i$.  Since $T$ is compact we may replace $t_i$ with a convergent subsequence, and define $t = \lim t_i$.  We claim that $t \in H^0(H_\infty, V)$.  Indeed, for any $i$ we have  
\[\sigma t - t = \sigma(t  - t_i) - (t - t_i) + \sigma t_i - t_i.\]
For any $n > 0$, we can choose $i$ large enough so that $(t - t_i) \in p^nT$ and $\sigma t_i - t_i \in p^nT$, showing that $\sigma t  = t$.  By Proposition \ref{hodgetate}, $t = 0$, which contradicts the fact that $t = \lim t_i$ and $t_i \not \in pT$.   
\end{proof}



\begin{thebibliography}{12}

\bibitem[BDP1]{BDP1}  
M.\ Bertolini, H.\ Darmon, and K.\ Prasanna, Generalized Heegner cycles and $p$-adic Rankin $L$-series, {\it Duke Math.\ J.} {\bf 162} no.\ 6, (2013) 1033-1148.

\bibitem[BDP2]{BDP3}  
M.\ Bertolini, H.\ Darmon, and K.\ Prasanna, Chow-Heegner points on CM elliptic curves and values of $p$-adic $L$-functions, {\it Int.\ Math.\ Res.\ Not.\ IMRN} no.\ 3, (2014) 745-793.

\bibitem[BDP3]{BDP5}  
M.\ Bertolini, H.\ Darmon, and K.\ Prasanna, $p$-adic $L$-functions and the coniveau filtration on Chow groups, submitted.

\bibitem[Be]{bertrand}
D.\ Bertrand, Valeurs de fonctions th\^eta et hauteurs $p$-adiques, {\it Prog.\ Math., } Vol.\ 22 Birkh\"auser, Boston, 1982, pp. 1-11.  

\bibitem[Br]{hunter}
E.H.\ Brooks, Shimura curves and special values of $p$-adic $L$-functions, {\it Int.\ Math.\ Res.\ Not.\ IMRN}, doi: 10.1093/imrn/rnu062

\bibitem[BK]{BK}
S.\ Bloch and K.\ Kato, {\it $L$-functions and Tamagawa numbers of motives}, in: The Grothendieck Festschrift, Vol.\ I, Progress in Mathematics {\bf 86}, Birkh\"{a}user, Boston, Basel, 1990, pp. 330-400.

\bibitem[Ca]{francesc}
F.\ Castella, Heegner cycles and higher weight specializations of big Heegner points, {\it Math.\ Ann.\ } {\bf 356} (2013), no.\ 4, 1247-1282. 

\bibitem[CH]{hsieh}
F.\ Castella and M.\ Hsieh, Heegner cycles and $p$-adic $L$-functions, preprint. 


\bibitem[Co]{colmez}
P.\ Colmez, Fonctions L $p$-adiques, S\'eminaire Bourbaki,  Vol.\ 1998/99, {\it Ast\'erisque}  {\bf 266} 2000, Exp.\ 851.

\bibitem[C1]{BC}  
B.\  Conrad, Gross-Zagier revisited, in Heegner points and Rankin $L$-series, {\it Math.\ Sci.\ Res.\ Inst.\ Pub.} {\bf 49}, Cambridge Univ.\ Press, Cambridge, 2004, 67-163. 

\bibitem[C2]{conrad}
B. Conrad, Lifting global representations with local properties, preprint.

\bibitem[DN]{niziol}
F.\ D\'eglise and W.\ Niziol, On $p$-adic absolute hodge cohomology and syntomic coefficients, I, preprint.  


\bibitem[dS]{dS}
E.\ de Shalit, Iwasawa theory of elliptic curves with complex multiplication, {\it Perspectives in Math.} {\bf 3}, Orlando:  Academic Press (1987).  

\bibitem[D]{disegni}
D.\ Disegni, $p$-adic heights of Heegner points on Shimura curves, preprint.

\bibitem[E]{yara}
Y.\ Elias, On the Selmer group attached to a modular form and an algebraic Hecke character, preprint.




\bibitem[F]{Falt}  
G.\  Faltings, Crystalline cohomology and $p$-adic Galois representations, pp. 25-80 in Algebraic Analysis, Geometry, and Number Theory, the Johns Hopkins University Press (1989).  

\bibitem[G]{Gr}  
B.\ Gross, Arithmetic on elliptic curves with complex multiplication, LNM 776, Springer-Verlag, 1980.  

\bibitem[GZ]{GZ}  
B.\  Gross, D.\ Zagier, Heegner points and derivatives of $L$-series, {\it Invent.\ Math.} {\bf 84}
(1986),  225-320.

\bibitem[H]{Hida1}  
H.\  Hida, A $p$-adic measure attached to the zeta functions associated with two elliptic modular forms. I, {\it Invent.\ Math.} {\bf 79}
(1985),  159-195.


\bibitem[Ho]{BHiwa}  
B.\  Howard, The Iwasawa theoretic Gross-Zagier theorem, {\it Compositio Math.\ } {\bf 141}, no.\ 4 (2005), 811-846

\bibitem[I]{Iwan}  
H.\ Iwaniec, Topics in classical automorphic forms, {\it Grad.\ Studies in Math.}, {\bf 17} 1997.  

\bibitem[KM]{KM}  
N.\ Katz and B.\ Mazur, Arithmetic Moduli of Elliptic Curves, {\it Ann.\ of Math.\ Studies} {\bf 108}, Princeton University Press, 1985.  

\bibitem[K]{kob}  
S.\ Kobayashi, The $p$-adic Gross-Zagier formula for elliptic curves at supersingular primes, {\it Invent.\ Math.\ } {\bf 191} (2013), no.\ 3, 527-629.

\bibitem[LZZ]{lzz}
Y.\ Liu, S.\ Zhang, and W.\ Zhang, On $p$-adic Waldspurger formula, preprint.  

\bibitem[M]{Miy}  
T.\ Miyake, Modular Forms, Springer-Verlag, 1989.  

\bibitem[N1]{NekEuler}  
J.\  Nekov\'{a}\v{r}, Kolyvagin's method for Chow groups of Kuga-Sato varieties, {\it Invent.\  Math.} {\bf 107}
(1992),  99-125.

\bibitem[N2]{Nekhts}  
J.\  Nekov\'{a}\v{r}, On $p$-adic height pairings, in: S\'eminaire de th\'eorie des numbres de Paris 1990/91,  {\it Progress in Math.} {\bf 108}, (1993),  (David, S., ed.), 127-202.

\bibitem[N3]{Nek}  
J.\  Nekov\'{a}\v{r}, On the $p$-adic height of Heegner cycles, {\it Math.\  Ann.} {\bf 302}
(1995), no.\ 4,  609-686.

\bibitem[N4]{Nek2}  
J.\  Nekov\'{a}\v{r}, $p$-adic Abel-Jacobi maps and $p$-adic heights, The arithmetic and geometry of algebraic cycles (Banff, AB, 1998), 367-379, CRM Proc. Lecture Notes, 24, {\it Amer.\  Math.\ Soc.}, Providence, RI, 2000.

\bibitem[Og]{Ogg}  
A.\ Ogg, Modular forms and Dirichlet series, Benjamin, (1969). 

\bibitem[Ol]{ols}  
M.\  Olsson, On Faltings' method of almost \'etale extensions, in {\it Algebraic Geometry --- Seattle 2005} part 2, Proc.\ Sympos.\ Pure Math.\ {\bf 80}, Amer.\ Math.\ Soc., Providence, 2009, 811-936. 

\bibitem[P]{petrequin}
D.\ P\'etrequin, Classes de Chern et classes de cycles en cohomologie rigide, {\it Bull.\ Soc.\ Math.\ France} {\bf 131} (1) (2003), 59-121.  

\bibitem[PR1]{PR1}  
B.\ Perrin-Riou, Points de Heegner et d\'{e}riv\'{e}es de fonctions $L$ $p$-adiques, {\it Invent.\ Math.} {\bf 89}
(1987),  455-510.

\bibitem[PR2]{PR2}  
B.\ Perrin-Riou, Fonctions $L$ $p$-adiques associ\'{e}es \`{a} une forme modulaire et \`{a} un corps quadratique imaginaire, {\it J.\ London Math.\ Soc.\ } {\bf 38} (1988),  1-32.

\bibitem[PR3]{PRbook}  
B.\ Perrin-Riou, $p$-adic $L$-functions and $p$-adic representations, {\it SMF/AMS Texts Monogr.\ } {\bf 3}, Amer.\ Math.\ Soc., Providence, and Soc.\ Math.\  France, Paris, 2000.

\bibitem[R]{Roh}  
D.\ Rohrlich, Root numbers of Hecke $L$-functions of CM fields, {\it Amer.\  J.\ Math.} {\bf 104}
(1982),  517-543.

\bibitem[Sc]{Sch}  
A.\ Scholl, Motives for modular forms, {\it Invent.\  Math.} {\bf 100}
(1990),  419-430.

\bibitem[Sh]{shiho}
A.\ Shiho, Crystalline fundamental groups. II. Log convergent cohomology and rigid cohomology, {\it J.\ Math.\ Sci.\ Univ.\ Tokyo} {\bf 9} (2002), 1-163.

\bibitem[ST]{ST}  
J.P.\ Serre, J.\ Tate, Good reduction of abelian varieties, {\it Ann.\ of Math.} {\bf 88}
(1968),  492-517.

\bibitem[Wa]{Wall}  
L.\ Walling, The Eichler commutation relation for theta series  with spherical harmonics, {\it Acta Arith.} {\bf 63}
(1993), no.\  3, 233-254.

\bibitem[Wi]{Wi}  
A.\ Wiles, On ordinary $\lambda$-adic representations associated to modular forms, {\it Invent.\ Math.} {\bf 94}
(1988),  529-573.

\bibitem[Z]{Zhang}  
S.W\ Zhang, Heights of Heegner cycles and derivatives of $L$-series, {\it Invent.\ Math.} {\bf 130}
(1997),  99-152.


\end{thebibliography}
\end{document}